\documentclass[sts]{imsart}
\usepackage[sort,round]{natbib}

\usepackage{amsmath, amssymb, color}
\usepackage{amsthm} 
\usepackage[colorlinks,citecolor=blue,urlcolor=blue]{hyperref}
\usepackage{xcolor,colortbl}
\RequirePackage{hypernat}
\usepackage{cleveref}
\usepackage{color}
\usepackage{amsfonts, fancybox}
\usepackage{amssymb}
\usepackage[american]{babel}
\usepackage{psfrag,color}
\usepackage{dcolumn}
\usepackage[format=hang, justification=justified, singlelinecheck=false]{subcaption}
\usepackage{flafter, bm, dsfont}
\usepackage[section]{placeins}
\usepackage[final]{showkeys}

\usepackage{multirow}
\usepackage{color}
\usepackage{amsmath}
\usepackage{float}
\usepackage{mathrsfs}
\usepackage{amsfonts}
\usepackage{dsfont}
\usepackage{amssymb}
\usepackage{wrapfig}

\usepackage{amssymb, latexsym}

\usepackage{graphicx}
\usepackage{epstopdf}
\usepackage{tikz}
\usepackage{pgfplots}
\usepgfplotslibrary{fillbetween}

\usepackage{color}

\usepackage{algorithm}
\usepackage[noend]{algcompatible}
\newtheorem{defin}{Definition}

\newtheorem{prop}{Proposition}
\newtheorem{lemma}{Lemma}

\newcommand{\cB}{\mathcal{B}}
\newcommand{\cC}{\mathcal{C}}
\newcommand{\cD}{\mathcal{D}}
\newcommand{\cE}{\mathcal{E}}

\newcommand{\cM}{\mathcal{M}}

\newcommand{\cP}{\mathcal{P}}
\newcommand{\cQ}{\mathcal{Q}}

\newcommand{\R}{{\rm I}\kern-0.18em{\rm R}}
\newcommand{\h}{{\rm I}\kern-0.18em{\rm H}}
\newcommand{\PP}{{\rm I}\kern-0.18em{\rm P}}
\newcommand{\E}{{\rm I}\kern-0.18em{\rm E}}
\newcommand{\Z}{{\rm Z}\kern-0.18em{\rm Z}}
\newcommand{\1}{{\rm 1}\kern-0.24em{\rm I}}
\newcommand{\N}{{\rm I}\kern-0.18em{\rm N}}

\newcommand*{\ep}{\varepsilon}
\newcommand*{\defeq}{:=}
\newcommand*{\dd}{\,\mathrm{d}}

\usepackage{mathtools}
\newtheorem{theorem}{Theorem}
\newtheorem{cor}{Corollary}

\newtheorem{assume}{Assumption}
\theoremstyle{definition}

\crefname{assume}{Assumption}{Assumptions}
\crefname{appendix}{Supplemental Appendix}{Supplemental Appendices}

\newcommand{\ud}{\mathrm{d}}

\DeclareMathOperator*{\diam}{diam}
\DeclareMathOperator*{\supp}{supp}
\newcommand*{\eqdef}{=:}
\newcommand{\dens}{\cD}
\newcommand{\bes}[3]{\cB^{#1}_{#2,#3}}
\newcommand{\besspace}{\bes{s}{p}{q}(L)}
\newcommand{\bounbes}{\bes{s}{p}{q}(L; m)}

\newcommand{\indic}[1]{\mathds{1}_{#1}}
\newcommand{\hol}[2]{\cC^{#1}(#2)}
\newcommand{\triplenorm}[1]{{\left\vert\kern-0.25ex\left\vert\kern-0.25ex\left\vert #1
    \right\vert\kern-0.25ex\right\vert\kern-0.25ex\right\vert}}
\newcommand{\RR}{\mathbb{R}}
\newcommand{\EE}{\mathbb{E}}
\newcommand*{\p}{\mathbb P}

\begin{document}
\begin{frontmatter}
	\title{Minimax estimation of smooth densities in Wasserstein distance}
	\runtitle{Density estimation in Wasserstein distance}

\begin{aug}
\author{\fnms{Jonathan}~\snm{Niles-Weed}\thanksref{t1}\ead[label=weed]{jnw@cims.nyu.edu	}}
\and
\author{\fnms{Quentin}~\snm{Berthet}\thanksref{t2}\ead[label=berthet]{qberthet@google.com}}

\affiliation{New York University\\ Google Research \& University of Cambridge}
\thankstext{t1}{Parts of this research were conducted while JNW was at the Institute for Advanced Study and at Massachusetts Institute of Technology, supported by the Josephine de Karman Fellowship.}
\thankstext{t2}{Most of this work was conducted while QB was at University of Cambridge and was supported in part by The Alan Turing Institute under the EPSRC grant EP/N510129/1.}

	\address{{Jonathan Niles-Weed}\\
		{Courant Institute of Mathematical Sciences} \\
		{New York University}\\
		{251 Mercer Street}\\
		{New York, NY 10012-1185, USA}\\
		\printead{weed}
	}

	\address{{Quentin Berthet}\\
	{Google Research} \\
	{8 Rue de Londres} \\
	{75009 Paris} \\
	{France} \\
		\printead{berthet}
	}

\runauthor{Niles-Weed and Berthet}
\end{aug}

	\begin{abstract}{}
	We study nonparametric density estimation problems where error is measured in the Wasserstein distance, a metric on probability distributions popular in many areas of statistics and machine learning.
We give the first minimax-optimal rates for this problem for general Wasserstein distances, and show that, unlike classical nonparametric density estimation, these rates depend on whether the densities in question are bounded below.
Motivated by variational problems involving the Wasserstein distance, we also show how to construct discretely supported measures, suitable for computational purposes, which achieve the minimax rates.
Our main technical tool is an inequality giving a nearly tight dual characterization of the Wasserstein distances in terms of Besov norms.
	\end{abstract}

\end{frontmatter}

\section{Introduction}

Optimal transport is a fundamental problem in geometry, optimization, and analysis~\citep{Mon81,Kan42,Vil08}, with increasing applications in statistics.
A great deal of recent work in the machine learning community has shown that optimal transport can be used to develop successful empirical methodologies~\citep[see, e.g.,][for a survey]{PeyCut17} at a relatively low computational cost~\citep{Cut13,AltWeeRig17}.
This raises the need for a rigorous theoretical understanding of the statistical properties of optimal transport.

The Wasserstein distances are a metric on probability distributions defined using optimal transport.
A metric is defined by asking how one can transport mass with distribution $\mu$ to have another distribution $\nu$, with minimal global transport cost. This problem also has the probabilistic interpretation, known as the Monge--Kantorovich formulation, of finding a joint distribution $\pi$ minimizing a cost for variables $X$ and $Y$ with given marginals. The Wasserstein distance emerges as the minimum value of this problem, and creates a natural tool to compare distributions, with $W_p$ corresponding to the $\|\cdot\|^p$ transport cost:
\begin{equation}\label{wp_def}
W_p^p(\mu,\nu) 
= \inf_{\pi \in \mathcal{M}(\mu,\nu)} \int \|x-y\|^p \ud \pi(x,y)\, ,
\end{equation}
where the set $\cM(\mu, \nu)$ denotes the set of joint measures with marginals $\mu$ and $\nu$, respectively.
This definition naturally extends to positive measures $\mu$ and $\nu$ with the same total mass.

In many modern applications, a Wasserstein distance is used as a loss function in an optimization problem over measures. Solving such problems involves optimizing functionals of the form $\nu \mapsto W_p(\nu, \mu)$ where $\mu$ is unknown.
Given $n$ i.i.d.~samples from $\mu$, much of the statistics literature adopts the plug-in approach and focuses on using the empirical distribution $\hat \mu_n$ to obtain the estimated functional $\nu \mapsto W_p(\nu, \hat \mu_n)$.
In this case, the rates of convergence are of order $n^{-1/d}$, and the sample size required for a particular precision is exponential in the dimension, a phenomenon known as the {\em curse of dimensionality}.
Moreover, it is known that this exponential dependence is tight, in the sense that no better estimate is available in general~\citep{SinPoc18}.

Our work adopts a different approach to show that the plug-in estimator is suboptimal for measures possessing a smooth density.
Estimating the density of a distribution, based on independent samples, is one of the fundamental problems of statistics. The usual goal in these problems is to produce an estimate $\tilde f$ which is as close as possible to the unknown density $f$, measured either at one point of the sample space, or in $L_p$ norm.
In this line of work, $f$ is usually assumed to belong to a large, nonparametric class defined via smoothness or regularity conditions, and typically the rates obtained in this setting show that sufficient smoothness can substantially mitigate the curse of dimensionality.
This is the subject of a wide literature on \emph{nonparametric density estimation}~\citep[see, e.g][]{Tsy09}.
In this work, we follow the same philosophy and derive similar rates for $W_p$ distances, over Besov classes of densities $\bes{s}{p}{q}$. We likewise show that the smoothness parameter $s$ improves the optimal exponent of $n$ in the Wasserstein setting.

The key challenge in proving tight rates for nonparametric density estimation in the Wasserstein distances is that the distances $W_p$ are both non-linear and non-local.
Unlike $L_p$ norms or Besov norms, commonly considered in the nonparametric estimation literature, the $W_p$ distance is not a norm on the space of probability measures when $p > 1$.
Moreover, it is not even ``norm-like'': we show (\cref{thm:not_norm}) that when $p > 1$, the Wasserstein distance $W_p$ cannot be controlled by any reasonable norm on densities.
Moreover, the optimization problem~\eqref{wp_def} that defines the Wasserstein distance is highly non-local: modifications to $\mu$ and $\nu$ in a small neighborhood can change the global structure of the optimal coupling $\pi$.
This makes analyzing the behavior of an estimator challenging.

Algorithmic aspects are an important part of optimal transport problems. For practical applications, the proposed estimates must therefore also be computationally tractable. We describe a method inspired by the parametric bootstrap~\citep{EfrTib93} to produce computationally tractable atomic estimators from any estimator that outperforms the empirical distribution, under minimal assumptions. We study the computational cost of this method, compared to the cost of using the empirical distribution with $n$ atoms, and exhibit a trade-off between computational cost and statistical precision.

\subsection{Prior work}
The question of establishing minimax rates for estimation in Wasserstein distances has been examined in several recent works. \citet{SinPoc18} established that, in the absence of smoothness assumptions, the empirical distribution $\hat \mu$ is rate optimal in a variety of examples. Their proof relies on a dyadic partitioning argument \citep[see, e.g.][]{WeeBac18}, and does not appear to extend to the smooth case.
Closer to our setting, under a smoothness assumption on the density of $\mu$, \citet{Lia17} and \citet{SinUppLi18} showed minimax rates of convergence for the Wasserstein-1 distance. To obtain these rates, these works focus on the dual form of $W_1$:
\begin{equation*}
W_1(\mu, \nu) = \sup_{f \in \mathrm{Lip}} \int f (\mathrm{d} \mu - \mathrm{d} \nu)\,,
\end{equation*}
where the supremum is taken over all 1-Lipschitz functions. This dual formulation puts the Wasserstein-1 distance into the category of \emph{integral probability metrics}~\citep{Mul97}, for which both \citet{Lia17} and \citet{SinUppLi18} obtain general results. It has been shown that choosing functions which are smoother than Lipschitz in this definition can result in improved rates of convergence for empirical measures~\citep{Klo18}. Crucially, the metric $W_p$ for $p > 1$ is \emph{not} an integral probability metric. Establishing sharp rates for general Wasserstein distances therefore requires different techniques.
A separate line of work has focused instead on modifying the definition of the Wasserstein distance to include a regularizing term based on the mutual information of the coupling.
It has been shown that this definition enjoys improved convergence rates relative to the unregularized version~\citep{GenChiBac18,MenNil19}.

Our proofs rely on establishing control of Wassserstein distances by Besov norms of negative smoothness. Similar results have been obtained elsewhere under different conditions.
\citet{ShiJac08} showed that the optimal transportation distance with cost $\|\cdot\|^p$ for $0 < p < 1$ can be characterized explicitly via an expression involving wavelet coefficients, which implies that these distances agree with a particular Besov norm (see \cref{sec:besov_introduction}); however, their proof technique does not extend to $p > 1$.
\citet{Loe06} \citep[see also][]{MauRouSan10} showed that the Wasserstein-2 distance between measures with densities bounded above dominates a negative Sobolev norm, and \citet{Pey18} extended this result to show that $W_2$ is in fact \emph{equivalent} to such a norm when the densities are in addition bounded below.
Similar connections have been developed by~\citet{LeeCoi15} and~\citet{Led17}.
To our knowledge, ours is the first result to establish a connection to Besov norms of negative smoothness and general Wasserstein distances.

The use of wavelet estimators for density estimation has a long history in nonparametric statistics~\citep{KerPic92,HarKerPic98,DonJohKer96,DouLeo90,Wal92}. However, while wavelets have been used for computational purposes in the optimal transport community~\citep{CheIweChi12,ShiJac08,DomAngTan08,RabPeyDel11}, the statistical properties of wavelet estimators with respect to Wasserstein distances have remained largely unexplored.
 
\section{Main results}

\subsection{Problem description and preliminaries}
Our observation consists of an i.i.d.~sample of size $n$ drawn from a probability measure on $\RR^d$ with smooth density $f$.
Our goal is to compute an estimator~$\tilde \mu_n$ that is close to $\mu_f$ in expected \emph{Wasserstein distance}.
As noted above, such an estimator can serve as a proxy for $\mu_f$ in statistical and computational applications.
While estimation of the density $f$ in norms such as $L_p$ is a well studied problem in nonparametric statistics~\citep{Tsy09}, such estimates do not readily lend themselves to guarantees in Wasserstein distance.

For technical reasons, we restrict ourselves to the case of compactly supported measures.
We further assume that this support is contained in $\Omega \defeq [0, 1]^d$.
The general case can be reduced to this one by a rescaling argument, as long as the statistician can obtain \emph{a priori} bounds on the size of the support of the measure.
We do not address the question of obtaining such bounds in this work, but note that they are often available in practice.

\subsubsection{Wavelets and Besov spaces}\label{sec:besov_introduction}
We direct the reader to~\citet{HarKerPic98} and~\citet{Mey90} for an introduction to the theory of wavelets.
In brief, we assume the existence of sets $\Phi$ and~$\Psi_{j}$ for $j \geq 0$ of functions in $L_2(\Omega)$ satisfying the standard requirements of a wavelet basis. (See \cref{app:wave} for our precise assumptions.) 

Wavelets can be used to characterize the Besov spaces $\bes{s}{p}{q}(\Omega)$. We follow the approach of~\citet{Coh03} for defining such spaces on bounded domains.
Suppose $s > 0$ and $p, q \geq 1$, and let $n > s$ be an integer.
Given $h \in \RR^d$, set
\begin{align*}
\Delta_h^1 f(x) & \defeq f(x+h) - f(x)\\
\Delta_h^k f(x) & \defeq \Delta_h^1(\Delta_h^{k-1}f)(x) \quad \forall 1< k \leq n\,,
\end{align*}
where these functions are defined on $\Omega_{h, n} \defeq \{x \in \Omega: x + nh \in \Omega\}$.
For $t > 0$, we then define
\begin{equation*}
\omega_n (f, t)_p = \sup_{\|h\| \leq t} \|\Delta_h^n f\|_{L_p(\Omega_{h, n})}\,.
\end{equation*}
The function $\omega_n$ measures the order-$n$ smoothness of $f$ in $L_p$.
Finally, we define the space $\bes{s}{p}{q}(\Omega)$ to be the set of functions for which the quantity
\begin{equation*}
\|f\|'_{\bes{s}{p}{q}} \defeq \|f\|_{L_p} + \left\|(2^{sj} \omega_n(f, 2^{-j})_p)_{j \geq 0}\right\|_{\ell_q}
\end{equation*}
is finite.

Assuming that the elements of $\Phi$ and $\Psi_j$ have $r$ continuous derivatives for $r > s$ and that polynomials of degree up to $\lfloor s \rfloor$ lie in the span of $\Phi$, the norm $\|\cdot\|'_{\bes{s}{p}{q}}$ is equivalent to a sequence norm based on wavelet coefficients.
Given $f \in L_p(\Omega)$, denote by $\alpha = \{\alpha_\phi\}_{\phi \in \Phi}$ the vector defined by $\alpha_\phi \defeq \int f \phi$ and by $\beta_j = \{\beta_\psi\}_{\psi \in \Psi_j}$ the vector whose entries are given by $\beta_\psi \defeq \int f \psi$.
Then $\|\cdot\|'_{\bes{s}{p}{q}}$ is equivalent to $\|\cdot\|_{\bes{s}{p}{q}}$ defined by
\begin{equation}\label{eq:bes_def}
\|f\|_{\bes{s}{p}{q}} \defeq \|\alpha\|_{\ell_p} + \left\|2^{js} 2^{dj(\frac 12 - \frac 1p)} \|\beta_j\|_{\ell_p} \right\|_{\ell_q}\,.
\end{equation}

This expression can then be used directly to define a norm when $s < 0$~\citep[see][Theorem 3.8.1]{Coh03}, as long as the elements of $\Phi$ and $\Psi_j$ have $r$ continuous derivatives for $r > |s|$ and polynomials of degree up to $\lfloor |s| \rfloor$ lie in the span of $\Phi$.
In what follows, we therefore adopt~\eqref{eq:bes_def} as our primary definition and assume throughout that the wavelet system has sufficient regularity that the equivalence of $\|\cdot\|'_{\bes{s}{p}{q}}$ and $\|\cdot\|_{\bes{s}{p}{q}}$ holds.

\subsubsection{Notation}

The quantities $C$ and $c$ will refer to constants whose value may change from line to line. \textbf{All constants throughout may depend on the choice of wavelet system and the dimension.} Since we are interested in establishing optimal rates of decay with respect to the exponent (i.e., finding $\gamma$ such that the rate $n^{-\gamma}$ holds), we leave finer control on dimension-dependent constants to future work.
We freely use the notation $a \lesssim b$ to indicate that there exists a positive constant $C$ for which $a \leq C b$ holds. Again, such constants may depend on the multiresolution and dimension.
The notation $a \asymp b$ indicates that $a \lesssim b$ and $b \lesssim a$.

Given $a, b \in \RR$, we denote by $a \wedge b$ and $a \vee b$ the minimum and maximum of $a$ and $b$, respectively.
The symbol $\|\cdot\|$ denotes the Euclidean norm on $\RR^d$.

We set denote by $\dens(K)$ the set of probability density functions on $\Omega$.
Since the measures we consider are absolutely continuous with respect to the Lebesgue measure, we do not distinguish between a probability measure and its density.
In particular, we write $W_p(f, g)$ for the $p$-Wasserstein distance between the measures with densities $f$ and $g$.

\subsection{Minimax estimation of smooth densities}\label{sec:smooth}
In this section, we give our main statistical results on the problem of estimating densities in Wasserstein distance.
These results reveal several phenomena: (i) the minimax rate of estimation can improve significantly for smooth densities, and (ii) the optimal rates depend strongly on whether the density in question is bounded away from $0$. Indeed, we show that the optimal rate for general densities is strictly worse than the corresponding rate for densities bounded below, no matter the smoothness.
While the first phenomenon is well known in nonparametric statistics, the second phenomenon does not occur in classical density estimation problems. As we explore further below, this behavior is fundamental to the Wasserstein distances.

We define two classes of probability densities on $\Omega$.
Given $m, L > 0$, set
\begin{align*}
\besspace & \defeq \{f \in L_p(\Omega) : \|f\|_{\bes{s}{p}{q}} \leq L, \int f = 1, f \geq 0\} \\
\bounbes & \defeq \besspace \cap \{f : f \geq m\}\,.
\end{align*}
We note that if $s$ is sufficiently large and $L$ is sufficiently small then in fact $\besspace \subseteq \bounbes$ for $m$ a constant.
We assume throughout that $m < 1$, since when $m \geq 1$, the class $\bounbes$ is trivial.

\subsubsection{Bounded densities}
Our first result gives an upper bound on the rate of estimation for functions in $\bounbes$.

\begin{theorem}\label{thm:estimation_ub}
For any $p \geq 1$ and $s \geq 0$, there exists an estimator $\hat f$ such that for any $p \leq p' < \infty$, $1 \leq q \leq \infty$ and $m > 0$, the estimator satisfies
\begin{equation*}
\sup_{f \in \bes{s}{p'}{q}(L; m)} \E W_p(f, \hat f) \lesssim \left\{ \begin{array}{ll}
n^{- \frac{1 + s}{d + 2s}} & d \geq 3 \\
n^{-1/2} \log n  & d = 2 \\
n^{-1/2} & d = 1\,.
\end{array}\right.
\end{equation*}
\end{theorem}
The upper bound in \cref{thm:estimation_ub} is achieved by a wavelet estimator.
As $s$ ranges between $0$ and $\infty$, the upper bound interpolates between the dimension-dependent rate $n^{-1/d}$ and the fully parametric rate $n^{-1/2}$.

The estimator proposed in \cref{thm:estimation_ub} is adaptive to the lower bound $m$, but not to the smoothness $s$.
An adaptive version of this estimator can be obtained at the price of an extra logarithmic factor.

\begin{theorem}\label{thm:estimation_ub_adaptive}
For any $p \geq 1$ and $L > 0$, there exists an estimator $\hat f^\circ$ such that for any $p \leq p' < \infty$, $1 \leq q \leq \infty$, $m > 0$, and $s \geq 0$, the estimator satisfies
\begin{equation*}
\sup_{f \in \bes{s}{p'}{q}(L;m)} \E W_p(f, \hat f^\circ) \lesssim \left\{ \begin{array}{ll}
n^{- \frac{1 + s}{d + 2s}} \log n & d \geq 3 \\
n^{-1/2} (\log n)^2  & d = 2 \\
n^{-1/2} & d = 1\,.
\end{array}\right. 
\end{equation*}
\end{theorem}
Though this estimator is adaptive to $s$, $p'$, and $q$, using it requires an \emph{a priori} estimate of $L$ (or, more specifically, the $L_p$ norm of $f$).
A similar phenomenon is present when designing adaptive density estimators in classical nonparametric statistics~\citep{DonJohKer96}.

Our lower bounds nearly match the upper bounds proved in Theorem~\ref{thm:estimation_ub}, up to a logarithmic factor in the $d = 2$ case.
\begin{theorem}\label{thm:estimation_lb}
For any $p,p',q \geq 1$, and $s \geq 0$,
\begin{equation*}
\adjustlimits\inf_{\tilde \mu} \sup_{f \in \bes{s}{p'}{q}(L; m)} \E W_p(f, \tilde \mu) \gtrsim \left\{ \begin{array}{ll}
n^{- \frac{1 + s}{d + 2s}} & d \geq 2 \\
n^{-1/2}  & d = 1\,,
\end{array}\right.
\end{equation*}
where the infimum is taken over all estimators $\tilde \mu$ based on $n$ observations.
\end{theorem}

Our core technical contribution is the following result, which establishes a connection between Wasserstein distances and Besov norms of negative smoothness.
\Cref{thm:besov_wp} allows us to bypass the difficulties of the Wasserstein distance by bounding instead a nearly equivalent norm.
\begin{theorem}\label{thm:besov_wp}
Let $p \in [1, \infty)$.
Let $f, g$ be two densities in $L_p([0, 1]^d)$, and assume $M \geq f(x) \vee g(x) \geq m > 0$ for almost every $x \in [0, 1]^d$.
Then
\begin{equation*}
M^{-1/p'} \|f - g\|_{\cB^{-1}_{p, \infty}} \lesssim W_p(f, g) \lesssim m^{-1/p'} \|f - g\|_{\cB^{-1}_{p, 1}}\,,
\end{equation*}
where $\frac 1p + \frac{1}{p'} = 1$.
\end{theorem}
\Cref{thm:besov_wp} can be viewed as a partial extension of the dual formulation of~$W_1$ to $W_p$ for $p > 1$.
Indeed, the inclusions
$\bes{1}{\infty}{1} \subseteq \mathrm{Lip} \subseteq \bes{1}{\infty}{\infty}$, where $\mathrm{Lip}$ is the space of bounded Lipschitz functions, imply $\|f - g\|_{\cB^{-1}_{1, \infty}} \lesssim W_1(f, g) \lesssim \|f - g\|_{\cB^{-1}_{1, 1}}$.
\Cref{thm:besov_wp} establishes the analogous result when $p > 1$, but only when the densities $f$ and $g$ are bounded.
A proof of this theorem appears in \cref{sec:wasbes}.

We prove \cref{thm:estimation_ub,thm:estimation_ub_adaptive,,thm:estimation_lb} in \cref{sec:bounded}.

\subsubsection{Unbounded densities}
Obtaining minimax rates when the densities are no longer bounded below is significantly more challenging since, as \cref{thm:not_norm} below makes clear, it is no longer possible to control $W_p$ by \emph{any} function norm in the absence of a lower bound on the densities in question.
Moreover, the statistical properties of the estimation problem also change markedly: surprisingly, the density estimation problem over the class $\besspace$ is strictly harder than the corresponding problem over $\bounbes$, even under a smoothness assumption.
We prove the following lower bound.

\begin{theorem}\label{thm:unbounded_lb}
For any $p,p',q \geq 1$, and $s \geq 0$, if $L$ is a sufficiently large constant, then
\begin{equation*}
\adjustlimits\inf_{\tilde \mu} \sup_{f \in \bes{s}{p'}{q}(L)} W_p(f, \tilde \mu) \gtrsim  \left\{\begin{array}{ll}

n^{-\frac{1 + s/p}{d + s}} & d -s \geq 2 p \\
n^{-1/2p} & d - s < 2p
\end{array}\right.
\end{equation*}
where the infimum is taken over all estimators $\tilde \mu$ based on $n$ observations.
\end{theorem}

Note that, when $p \geq 2$, this rate is worse than the upper bound given in \cref{thm:estimation_ub} for all $s > 0$ and $d \geq 1$.
This establishes that the class of densities bounded from below is strictly easier to estimate than the class of all densities, for all nontrivial smoothness parameters.

When $s \in [0, 1)$, and we consider the Holder class $\hol{s}{L}$, we can obtain a nearly matching upper bound, up to a logarithmic factor.
Moreover, the estimator we construct is a \emph{histogram}. This property enables the use of such an estimator in practical applications. We take up this point in \cref{sec:comp}.

\begin{theorem}\label{thm:unbounded_ub}
Assume $p \geq 2$.
For any $s \in [0, 1)$, there exists a histogram estimator $\hat f$ such that
\begin{equation*}
\sup_{f \in \hol{s}{L}} \E W_p(f, \hat f) \lesssim \left\{ \begin{array}{ll}
n^{- \frac{1 + s/p}{d + s}} & d - s > 2 p \\
n^{- \frac{1}{2p}} \log n  & d - s = 2p \\
n^{- \frac{1}{2p}}  & d - s< 2p\,.
\end{array}\right.
\end{equation*}
\end{theorem}

The proofs of both \cref{thm:unbounded_ub,thm:unbounded_lb} appear in \cref{sec:unbounded}.

\subsection{Computational aspects of smooth density estimation}
In many computational applications, it is significantly simpler to work with discrete measures supported on a finite number of points, since in general there is no closed form expression for the Wasserstein distance between continuous measures.
Unfortunately, the estimators presented in \cref{sec:smooth} are not of this form, so it is unclear whether smoothness of the underlying measure can be exploited in applications.
However, a simple argument shows that optimal rates can be achieved by \emph{resampling} from the smooth estimator we construct to obtain a discrete distribution supported on $M \geq n$ points which achieves the optimal rate for $s \in [0, 1)$.
We extract one simple result in this direction.
\begin{theorem}\label{thm:ub_comp}
For any $s \in [0, 1)$ and any $1 \leq p < \infty$, there exists an estimator $\bar \mu_{n, M}$, supported on $M = o(n^2)$ points, enjoying the same rate as in \cref{thm:unbounded_ub}, up to logarithmic factors.
Moreover, $\bar \mu_{n, M}$ can be computed in time $O(M)$.
\end{theorem}
Additional computational considerations along with a proof of \cref{thm:ub_comp} appear in \cref{sec:comp}. 
\section{Controlling the Wasserstein distance by Besov norms}\label{sec:wasbes}
The main goal of this section is a proof of \cref{thm:besov_wp}, which establishes that the Wasserstein distance between two measures on $\Omega = [0, 1]^d$ can be controlled by a Besov norm of the difference in their densities \emph{as long as their densities are bounded above and below}. We also establish that no analogous result can hold for arbitrary densities.
While we give upper and lower bounds, the Besov norms appearing in the two bounds do not agree. We do not know whether under some conditions the $W_p$ distance is in fact \emph{equivalent} to a particular Besov norm $\|\cdot\|_{\bes{-1}{p}{q}}$ for some $q \in [1, + \infty]$.

The results of this section are closely related to results of \citet{ShiJac08} and \citet{Pey18}, who established similar relations for $p < 1$ and $p = 2$, respectively.
In \cref{sec:bes_ub}, we show the upper bound of \cref{thm:besov_wp}, and in \cref{sec:bes_lb} we show the lower bound. \Cref{sec:gen_dens} establishes that there is no general relationship between Wasserstein distances and Besov norms once the assumption that the density is bounded away from $0$ is relaxed.

\subsection{Upper bound}\label{sec:bes_ub}
Let $f$ and $g$ be probability densities in $L_p(\Omega)$ for $p \in [1, \infty)$ with the following wavelet expansions.
\begin{equation}\label{eq:fg-expansion}
\begin{aligned}
f & = \sum_{\phi \in \Phi} \alpha_\phi \phi + \sum_{j \geq 0} \sum_{\psi \in \Psi_j} \beta_\psi \psi \\
g & = \sum_{\phi \in \Phi} \alpha'_\phi \phi + \sum_{j \geq 0} \sum_{\psi \in \Psi_j} \beta'_\psi \psi\,,
\end{aligned}
\end{equation}
where we assume (\cref{assume:regularity}) that constant functions lie in the span of $\Phi$.
For the upper bound, we do not need to assume any additional regularity---in particular, \cref{prop:wavelet_ub} holds for the Haar wavelet basis~\citep[see][]{Tri10}.
In principle, the expansions in~\eqref{eq:fg-expansion} hold only in $L_2$, but in fact convergence also holds in $L_p$ assuming that $f, g \in L_p(\Omega)$~\citep[Remark 8.4]{HarKerPic98}.

\begin{prop}\label{prop:wavelet_ub}
Let $1 \leq p < \infty$.
If $f(x) \vee g(x) \geq m > 0$ for almost every $x \in [0, 1]^d$, then
\begin{equation*}
W_p(f, g) \lesssim m^{- 1/p'} \Big( \|\alpha - \alpha'\|_{\ell_p} + \sum_{j \geq 0} 2^{-j} 2^{dj(\frac 12 - \frac 1 p)} \|\beta_j - \beta_j'\|_{\ell_p}\Big)\,,
\end{equation*}
where $\frac 1p + \frac{1}{p'} = 1$.
\end{prop}
By rescaling $f$ and $g$ by a positive real number, this proposition implies that the analogous claim holds for any two nonnegative functions $f$ and $g$ satisfying $\int_\Omega f = \int_\Omega g < \infty$.

\begin{proof}
We will follow a strategy originally developed by~\citet{Mos65} for the purpose of showing that all volume forms on a smooth, compact manifold are equivalent up to automorphism.
We define a vector field $V$ on $\Omega$ satisfying
\begin{align}
\nabla \cdot V & = f - g \label{eq:v_divergence}\\
\|V\|_{L^p} & \lesssim \|\alpha - \alpha'\|_{\ell_p} + \sum_{j \geq 0} 2^{-j} 2^{dj(\frac 12 - \frac 1 p)} \|\beta_j - \beta_j'\|_{\ell_p}\label{eq:v_norm}\,,
\end{align}
where the first condition is intended in the distributional sense that
\begin{equation*}
- \int_\Omega \nabla h \cdot V \dd x = \int_\Omega h (f - g) \dd x\,.
\end{equation*}
for all $h \in C^1(\Omega)$.
In particular, we require the boundary condition $V \cdot \mathbf{n} = 0$ on $\partial \Omega$, where $\mathbf{n}$ is an outward-pointing normal.
We defer the construction of this vector field to Proposition~\ref{prop:vlambda}, below.

To show the theorem, we appeal to the following characterization of the Wasserstein distance.
Denote by $K_\Omega$ the set of pairs of measures $(\rho, E)$ on $\Omega \times [0, 1]$ where $\rho$ is scalar valued and $E$ is vector valued.
\begin{theorem}[\citealp{BenBre00,Bre03}]
For any measures $\mu$ and $\nu$ on $\Omega$ and $p \in [1, \infty)$,
\begin{equation*}
W_p^p(\mu, \nu) = \inf_{(\rho, E) \in K_\Omega} \left\{\cB_p(\rho, E): \rho(\cdot, 0) = \mu, \rho(\cdot, 1) = \nu,  \partial_t \rho + \nabla_x \cdot E = 0 \right\}\,,
\end{equation*}
where
\begin{equation*}
\cB_p(\rho, E) \defeq \left\{\begin{array}{ll}\int_{\Omega \times [0, 1]} \left\| \frac{d E}{d \rho}(x, t)\right\|^p \dd \rho(x, t) & \text{ if $E \ll \rho$,} \\
+ \infty & \text{ otherwise.}
\end{array}\right.
\end{equation*}
\end{theorem}

Let us show how to prove the theorem.
We choose $\rho$ and $E$ absolutely continuous with respect to the Lebesgue measure on $\Omega \times [0, 1]$, and consequently identify them with their density.
First, set $\rho(x, t) = (1-\lambda(t)) f(x) + \lambda(t) g(x)$, where $\lambda: [0, 1] \to [0, 1]$ is defined by
\begin{equation*}
\lambda(x) \defeq \left\{
	\begin{array}{ll}
		\frac 12 (2 t)^p & \text{ if $t \leq 1/2$,} \\
		1 - \frac 12 (2 - 2t)^p & \text{ if $t > 1/2$.}
	\end{array}
	\right.
\end{equation*}
Clearly $\rho(\cdot, 0) = f(\cdot)$ and $\rho(\cdot, 1) = g(\cdot)$.
Moreover, we have the following lower bound.
\begin{lemma}\label{interp_lb}
Suppose $f(x) \vee g(x) \geq m$ for almost all $x \in \Omega$.
Then
\begin{equation*}
\rho(x, t) \geq \left\{
	\begin{array}{ll}
		\frac 12 (2t)^p m & \text{ if $t \leq 1/2$,} \\
		\frac 12 (2 - 2t)^p m & \text{ if $t > 1/2$}
	\end{array}
	\right.
\end{equation*}
for almost every $x \in \Omega$.
In particular, for $t \neq 1/2$, we have that $\rho(x, t) \geq \frac 12 m \cdot \left(\frac{\lambda'(t)}{p}\right)^{\frac{p}{p-1}}$.
\end{lemma}
We then define $E$ by
\begin{equation*}
E(x, t) = \lambda'(t) V(x) \quad \text{ for $t \in [0, 1] \setminus \{1/2\}$.}
\end{equation*}
Since $\nabla_x \cdot E = \lambda'(t) (f-g)$ for almost all $x \in \Omega$ and $t \in [0, 1]$, the pair $(\rho, E)$ defined in this way satisfies $\partial_t \rho + \nabla_x \cdot E = 0$ in the distributional sense.

For almost every $t \in [0, 1]$, we have the bound
\begin{equation*}
\left\| \frac{d E}{d \rho}(x, t)\right\|^p \rho(x, t) \leq \|V(x)\|^p \frac{\lambda'(t)^p}{\rho(x, t)^{p-1}} \lesssim \|V(x)\|^p m^{1-p}\,.
\end{equation*}
We obtain
\begin{align*}
W_p(f, g) & \leq \left(\int_{\Omega \times [0, 1]} \left\| \frac{d E}{d \rho}(x, t)\right\|^p \dd \rho(x, t)\right)^{1/p} \\
& \lesssim m^{1/p - 1} \|V\|_{L_p} \\
& \lesssim m^{-1/p'} \left(\|\alpha - \alpha'\|_{\ell_p} + \sum_{j \geq j_0} 2^{-j} 2^{dj(\frac 12 - \frac 1 p)} \|\beta_j - \beta_j'\|_{\ell_p}\right)\,,
\end{align*}
as claimed.
\end{proof}

All that remains is to establish the existence of the promised vector field $V$.
\begin{prop}\label{prop:vlambda}
There exists a vector field $V$ satisfying the requirements of~\eqref{eq:v_divergence} and~\eqref{eq:v_norm}.
\end{prop}
\begin{proof}
We will proceed by defining vector fields $V_\phi$ for each $\phi \in \Phi$ and $V_\psi$ for each $\psi \in \Psi_j$, $j \geq 0$ satisfying $\nabla \cdot V_\phi = \phi$ and $\nabla \cdot V_\psi = \psi$, along with appropriate boundary conditions.
The desired vector field $V$ will then be obtained as
\begin{equation}\label{eq:v_def}
V = \sum_{\phi \in \Phi} (\alpha_\phi - \alpha'_\phi) V_\phi + \sum_{j \geq 0} \sum_{\psi \in \Psi_j} (\beta_\psi - \beta'_\psi) V_\psi\,.
\end{equation}
An application of Fubini's theorem immediately yields that this definition satisfies~\eqref{eq:v_divergence} in the distributional sense.
We conclude by obtaining the desired estimate for $\|V\|_{L_p}$ to show~\eqref{eq:v_norm}

\paragraph{Definition of $V_\phi$ for $\phi \in \Phi$.}
Given $x \in \RR^d$, we write $x^{(i)}$ for the vector consisting of the first $i$ coordinates of $x$.
For each $1 \leq i \leq d$, define $\phi^{(i)}: \RR^i \to \RR$ by
\begin{equation*}
\phi^{(i)}(x^{(i)}) = \int_0^1 \dots \int_0^1 \phi(x^{(i)}, t_{i+1}, \dots, t_d) \dd t_{i+1} \dots \dd t_d\,,
\end{equation*}
and set $\phi^{(0)} = 0$.
We define $V_\phi$ componentwise as
\begin{equation*}
(V_\phi)_i(x)  = \int_0^{x_i} \phi^{(i)}(x^{(i-1)}, t_i) \dd t_i - x_i \phi^{(i-1)}(x^{(i-1)}) \quad \quad 1 \leq i \leq d\,.
\end{equation*}

We now verify that this definition satisfies the desired identity. The proof appears in \cref{app:proofs}.
\begin{lemma}\label{lem:vphi_divergence}
The field $V_\phi$ satisfies $\nabla \cdot V_\phi = \phi$.
Moreover, 
$\left(\sum_{\phi \in \Phi} (\alpha_\phi - \alpha'_\phi) V_\phi\right)\cdot \mathbf{n} = 0$ on the boundary of $[0, 1]^d$, where $\mathbf n$ is an outward-pointing normal.
\end{lemma}

\paragraph{Definition of $V_\psi$ for $\psi \in \Psi_j$, $j \geq 0$.}
We adopt essentially the same construction as above.
First, by \cref{assume:tensor}, $\psi$ can be written as $\bigotimes_{i=1}^d \psi_i$, where each $\psi_i$ is a univariate function. \Cref{assume:basis,assume:regularity} imply that $\int_{[0, 1]^d} \psi(x) \dd x= 0$, so there exists an index $k \in [d]$ such that $\int_{[0, 1]} \psi_k(x_k) \dd x_k = 0$.
We set
\begin{equation*}
(V_\psi)_k(x) = \int_0^{x_k} \psi_k(t) \dd t \cdot \prod_{i \neq k} \psi_i(x_i) \,,
\end{equation*}
and $(V_\psi)_i = 0$ for $i \neq k$.

\begin{lemma}\label{lem:vpsi_divergence}
The field $V_\psi$ satisfies $\nabla \cdot V_\psi = \psi$ and $V_\psi \cdot \mathbf{n} = 0$ on the boundary of $[0, 1]^d$, where $\mathbf{n}$ is an outward-pointing normal.
\end{lemma}
A proof appears in \cref{app:proofs}.

\paragraph{Norm estimates.}

We now obtain an estimate for $\|V\|_{L_p}$.
We require two lemmas, the proofs of which appear in \cref{app:proofs}.

\begin{lemma}\label{lem:vphi_norm_estimate}
For any sequence $\{\alpha_\phi\}_{\phi \in \Phi}$,
\begin{equation*}
\left\|\sum_{\phi \in \Phi} \alpha_\phi V_\phi\right\|_{L_p} \lesssim \|\alpha\|_{\ell_p}\,.
\end{equation*}
\end{lemma}

\begin{lemma}\label{lem:vpsi_norm_estimate}
For any sequence $\{\beta_\psi\}_{\psi \in \Psi_j}$,
\begin{equation*}
\left\|\sum_{\psi \in \Psi_j} \beta_\psi V_\psi\right\|_{L_p} \lesssim 2^{-j} 2^{dj(\frac 12 - \frac 1p)} \|\beta\|_{\ell_p}\,.
\end{equation*}
\end{lemma}

We obtain, for $V$ defined as in~\eqref{eq:v_def},
\begin{align*}
\|V\|_{L_p} & \leq \left\|\sum_{\phi \in \Phi} (\alpha_\phi - \alpha'_\phi) V_\phi\right\|_{L_p} + \sum_{j \geq 0}  \left\|\sum_{\psi \in \Psi_j} (\beta_\psi - \beta'_\psi) V_\psi\right\|_{L_p} \\
& \lesssim \|\alpha - \alpha'\|_{\ell_p} + \sum_{j \geq 0} 2^{-j} 2^{dj(\frac 12 - \frac 1p)} \|\beta_j - \beta_j'\|_{\ell_p}\,,
\end{align*}
which gives~\eqref{eq:v_norm}.
\end{proof}

\subsection{Lower bound}\label{sec:bes_lb}
We can prove a similar lower bound when $f$ and $g$ are bounded above.
Unlike the assumption that the densities are bounded below required for \cref{prop:wavelet_ub}, this assumption is relatively benign, insofar as it holds automatically for continuous densities on $[0, 1]^d$.
For \cref{prop:wavelet_lb}, we require the wavelets in~\eqref{eq:fg-expansion} to possess at least one continuous derivative (see \cref{assume:regularity}).

\begin{prop}\label{prop:wavelet_lb}
Let $1 \leq p < \infty$. If $f(x) \vee g(x) \leq M$ for almost every $x \in [0, 1]^d$, then
\begin{equation*}
W_p(f, g) \gtrsim M^{-1/p'} \left(\|\alpha - \alpha'\|_{\ell_p} + \sup_{j \geq †0}\left\{ 2^{-j} 2^{dj(\frac 12 - \frac 1 p)} \|\beta_j - \beta_j'\|_{\ell_p}\right\}\right)\,.
\end{equation*}
\end{prop}

\begin{proof}
We use the following fact, due to~\citet[Lemma~3.4 and Remark]{MauRouSan10}:
\begin{lemma}[\citealp{MauRouSan10}]\label{lem:sob_dual}
For all $h \in C^1(\Omega)$,
\begin{equation*}
\int_\Omega h (f - g) \dd x \leq M^{1/p'} \|\nabla h\|_{L_{p'}(\Omega)} W_p(f, g)\,.
\end{equation*}
\end{lemma}

Fix an index $j \geq 0$.
Let $h$ be a function of the form
\begin{equation*}
h = \sum_{\phi \in \Phi} \kappa_\phi \phi + \sum_{\psi \in \Psi_j} \lambda_\psi \psi\,,
\end{equation*}
for some vectors $\kappa$ and $\lambda$ satisfying $\|\kappa\|_{\ell_{p'}} \leq 1$ and $\|\lambda\|_{\ell_{p'}} \leq 2^{-j + dj(\frac 12 - \frac 1p)}$.

We require the following bound, whose proof appears in \cref{app:proofs}.
\begin{lemma}\label{lem:h_estimate}
If $\|\kappa\|_{\ell_{p'}} \leq 1$ and $\|\lambda\|_{\ell_{p'}} \leq 2^{-j + dj(\frac 12 - \frac 1p)}$, then $\|\nabla h\|_{L_{p'}(\Omega)} \lesssim 1$.
\end{lemma}

Applying Lemmas~\ref{lem:sob_dual} and~\ref{lem:h_estimate}, we obtain
\begin{equation*}
W_p(f, g) \gtrsim M^{-1/p'} \int_\Omega h(f - g) \dd x = M^{-1/p'} \left(\sum_{\phi \in \Phi} \kappa_\phi (\alpha_\phi - \alpha'_\phi) + \sum_{\psi \in \Psi_j} \lambda_\psi (\beta_\psi - \beta'_\psi)\right)\,.
\end{equation*}
Taking the supremum over $\kappa$ and $\lambda$ subject to the constraints $\|\kappa\|_{\ell_{p'}} \leq 1$ and $\|\lambda\|_{\ell_{p'}} \leq 2^{-j + dj(\frac 12 - \frac 1p)}$ implies
\begin{equation*}
W_p(f, g) \gtrsim M^{-1/p'} \left(\|\alpha - \alpha'\|_{\ell_p} +  2^{-j} 2^{dj(\frac 12 - \frac 1 p)} \|\beta_j - \beta_j'\|_{\ell_p}\right)\,,
\end{equation*}
and taking the supremum over $j \geq 0$ yields the claim.
\end{proof}

\subsection{Densities not bounded below}\label{sec:gen_dens}
We now show that no statement like \cref{prop:wavelet_ub} can hold if one of the densities is not bounded away from zero. Indeed, in this case, under mild assumptions, it is impossible to control $W_p(f, g)$ by any function norm when $p > 1$.
This stands in sharp contrast to the fact that, when $p = 1$, the dual formulation of $W_1$ implies that the Wasserstein distance is such a norm.

\begin{theorem}\label{thm:not_norm}
Let $\triplenorm{\cdot}$ be any norm on functions on $\Omega$, and suppose that there exists a function $h$ in~$L_1(\Omega)$, not identically zero, satisfying
\begin{itemize}
\item $\int_\Omega h = 0$
\item $\triplenorm{h} < \infty$
\item The sets $\overline{\{h > 0\}}$ and $\overline{\{h < 0\}}$ are disjoint.
\end{itemize}

Then for any $p > 1$,
\begin{equation*}
\sup_{f, g \in \dens(\Omega)} \frac{W_p(f, g)}{\triplenorm{f - g}} = \infty\,.
\end{equation*}
\end{theorem}
\begin{proof}
Let $h_+ \defeq h \vee 0$ and $h_- \defeq -(h \wedge 0)$, so that $h = h_+ - h_-$, and note that by assumption $\int_\Omega h_+ = \int_\Omega h_- \eqdef \rho > 0$.
For any $\lambda \in [0, 1]$, set 
\begin{align*}
f_\lambda & = \frac {1}{2 \rho} ((1+\lambda) h_+ + (1-\lambda) h_-) \\
g_\lambda & = \frac {1}{2 \rho} ((1- \lambda) h_+ + (1+\lambda) h_-)\,.
\end{align*}
Note that $f_\lambda, g_\lambda \in \dens(\Omega)$, and $\triplenorm{f - g} = \frac{\lambda}{\rho} \triplenorm{h}$.
On the other hand, since the compact sets $\overline{\{h > 0\}}$ and $\overline{\{h < 0\}}$ are disjoint, there exist two sets $S$ and $T$ and $c > 0$ such that $\supp(h_+) \subseteq S$ and $\supp(h_-) \subseteq T$ and $\|x - y\| \geq c$ for any $x \in S, y \in T$.
\Cref{lem:wp_separated_lb}, below, therefore implies $W_p(f, g) \geq c|\int_S f - g|^{1/p} = c \lambda^{1/p}$.

We obtain
\begin{equation*}
\sup_{f, g \in \dens(\Omega)} \frac{W_p(f, g)}{\|f - g\|} \geq \sup_{\lambda \in (0, 1)} \frac{W_p(f_\lambda, g_\lambda)}{\|f_\lambda - g_\lambda\|} \gtrsim \sup_{\lambda \in (0, 1)} \lambda^{1/p - 1} = \infty\,.
\end{equation*}
\end{proof}

\section{Wavelet estimation in for bounded densities}\label{sec:bounded}
In this section, we employ the results of \cref{sec:wasbes} to prove \cref{thm:estimation_ub,thm:estimation_ub_adaptive,,thm:estimation_lb}.
We show that the minimax rate over $\bounbes$ can be achieved by a wavelet estimator.

\subsection{Upper bound}
To prove Theorem~\ref{thm:estimation_ub}, we introduce the following estimator based on a wavelet expansion of regularity $r > \max\{s, 1\}$ (see \cref{assume:regularity}) truncated to level $J$, for some $J \geq 0$ to be chosen. Set
\begin{align*}
\tilde \alpha_\phi & \defeq \frac 1n \sum_{i} \phi(X_i) \quad \phi \in \Phi\\
\tilde \beta_\psi & \defeq \frac 1n \sum_{i} \psi(X_i) \quad \psi \in \Psi_j, 0 \leq j \leq J 
\end{align*}
and let $\tilde f \defeq \sum_{\phi \in \Phi} \tilde \alpha_\phi \phi + \sum_{0 \leq j \leq J} \sum_{\psi \in \Psi_j} \tilde \beta_\psi \psi$.
While such an estimator can already yield optimal rates in $L_p$~\citep{KerPic92}, $\tilde f$ may fail to be a probability density, in which case the quantity $W_p(f, \tilde f)$ is undefined.
We therefore focus on the estimator
\begin{equation*}
\hat f \defeq \min_{g \in \dens} \|g - \tilde f\|_{\bes{-1}{p}{1}}\,,
\end{equation*}
where $\dens$ is the set of probability densities on $\Omega$.
By construction, $\hat f$ is a density, so that $W_p(f, \hat f)$ is meaningful.

The following lemma shows that the quality of this estimator can be controlled by the distance between $f$ and $\tilde f$ in Besov norm.
\begin{lemma}\label{lem:wp_to_bes}
For any $f \in \bounbes$,
\begin{equation*}
W_p(f, \hat f) \lesssim \|f - \tilde f\|_{\bes{-1}{p}{1}}\,.
\end{equation*}
\end{lemma}
\begin{proof}
By assumption, $f$ is bounded below by $m$. Theorem~\ref{thm:besov_wp} therefore implies
\begin{equation*}
W_p(f, \hat f) \lesssim \|f - \hat f\|_{\bes{-1}{p}{1}} \leq \|f - \tilde f\|_{\bes{-1}{p}{1}} + \|\hat f - \tilde f\|_{\bes{-1}{p}{1}} \lesssim \|f - \tilde f\|_{\bes{-1}{p}{1}}\,,
\end{equation*}
where the final inequality uses the definition of $\hat f$ and the fact that $f \in \dens$.
\end{proof}

The proof of Theorem~\ref{thm:estimation_ub} now follows from standard facts in wavelet density estimation.
We require the following proposition, whose proof appears in \cref{app:proofs}.
\begin{prop}\label{prop:moment_bounds}
Let $f$ have wavelet expansion as in~\eqref{eq:fg-expansion}, and let $0 \leq j \leq J$.
If $n \geq 2^{dJ}$ and $p \geq 1$, then
\begin{align*}
\EE \|\alpha - \tilde \alpha\|_{\ell_p} & \lesssim \frac{1}{n^{1/2}}\\
\EE \|\beta_j - \tilde \beta_j\|_{\ell_p} & \lesssim \frac{2^{dj/p}}{n^{1/2}}\,.
\end{align*}
\end{prop}

We are now in a position to prove Theorem~\ref{thm:estimation_ub}.
\begin{proof}[Proof of Theorem~\ref{thm:estimation_ub}]
Denote by $f_J$ the projection of $f$ to $\mathrm{span}\left(\Phi \cup \left(\bigcup_{j \leq J} \Psi_j\right)\right)$, i.e..
\begin{equation*}
f_J = \sum_{\phi \in \Phi} \alpha_\phi \phi + \sum_{0 \leq j \leq J} \sum_{\psi \in \Psi_j} \beta_\psi \psi\,.
\end{equation*}
The assumption that $f \in \bes{s}{p'}{q}(L)$ implies by \cref{assume:size} that
$$2^{dj(\frac 12 - \frac 1 p)} \|\beta_j\|_{\ell_p} \lesssim 2^{dj(\frac 12 - \frac{1}{p'})} \|\beta_j\|_{\ell_{p'}} \lesssim 2^{-js}$$
for all $j \geq 0$; hence
\begin{equation*}
\|f - f_J\|_{\bes{-1}{p}{1}} = \sum_{j > J} 2^{-j} 2^{dj(\frac 12 - \frac 1 p)} \|\beta_j\|_{\ell_p} \lesssim 2^{-J(s + 1)}\,.
\end{equation*}
Lemma~\ref{lem:wp_to_bes} implies
\begin{align*}
\E W_p(f, {\hat f}) & \lesssim \E \|f_J - \tilde f\|_{\bes{-1}{p}{1}} + \|f - f_J\|_{\bes{-1}{p}{1}} \\
& \lesssim \E \|\alpha - \tilde \alpha\|_{\ell_p} + \sum_{0 \leq j \leq J} \sum_{\psi \in \Psi_j} 2^{-j}2^{dj(\frac 12 - \frac 1p)} \E \|\beta_j - \tilde \beta_j\|_{\ell_p} + 2^{-J(s + 1)} \\
& \lesssim \sum_{0 \leq j \leq J} 2^{-j} \left( \frac{2^{dj}}{n}\right)^{1/2} + 2^{-J(s + 1)}
\end{align*}
by \cref{prop:moment_bounds}, as long as $2^J \leq n^{1/d}$.
Choose $J$ so that $2^J \asymp n^{\frac{1}{d + 2s}}$.
If $d \geq 3$, the last term in the sum dominates and is of the same order as $2^{-J(s+1)} = n^{-{\frac{1+s}{d+2s}}}$.
If $d \leq 2$, the approximation term is negligible and the sum is of order $n^{-1/2}$ if $d =1$, or $n^{-1/2} \log n$ if $d = 2$.
\end{proof}

\subsection{An adaptive estimator}
By applying a truncation technique, we can also obtain an adaptive procedure which achieves the optimal rate (up to logarithmic factors) for any $s \geq 0$.

For any $j \geq 0$, set
\begin{align}\label{tau_j}
\tau_j & \defeq 8 (1+L) K 2^{dj/p} \frac{dj}{\sqrt n}\,,
\end{align}
where $K$ is a constant defined in \cref{bernstein_bound} depending on the wavelet system.
For each $j \geq 0$, define by $\tilde \beta_j$ the vector of empirical wavelet coefficients at level $j$, that is, the vector indexed by $\Psi_j$ whose coordinates are given by
\begin{equation*}
\tilde \beta_\psi \defeq \frac 1n \sum_{i=1}^n \psi(X_i)\,.
\end{equation*}
The threshold $\tau_j$ is chosen in such a way that $\|\tilde \beta_j - \beta_j\|_{\ell_p} \ll \tau_j$ for all $j \leq J$.
(See \cref{bernstein_bound}.)

We let $J$ be such that $2^J \asymp n^{1/d}$, and set
\begin{equation*}
\tilde f^\circ \defeq \sum_{\phi \in \Phi} \tilde \alpha_\phi \phi + \sum_{0 \leq j \leq J} \left(\sum_{\psi \in \Psi_j} \tilde \beta_\psi \psi\right) \indic{\|\tilde \beta_j\|_{\ell_p} \geq \tau_j}
\end{equation*}
Finally, in order to obtain an estimator that is a probability density, we defined the projected estimator $\hat f^\circ$ by
\begin{equation*}
\hat f^\circ \defeq \min_{g \in \cD} \|g - \tilde f^\circ\|_{\cB^{-1}_{p, 1}}\,.
\end{equation*}
The analysis of $\hat f^\circ$ is standard, and follows ideas for adaptive estimation developed by~\citet{DonJohKer96}.
We defer the details to \cref{adaptivity_bound}.

\subsection{Lower bound}
Our lower bound follows almost directly from the bound proved by~\citet{KerPic92} to establish minimax rates for density estimation in $L_p$ over Besov spaces.
By the monotonicity of the Wasserstein-$p$ distances in $p$, it suffices to prove the lower bound for $p = 1$.

\begin{proof}[Proof of \cref{thm:estimation_lb}]
Given an index $J$ to be specified and a vector $\ep \in \{\pm 1\}^{|\Psi_J|}$, we write
\begin{equation*}
f_\ep \defeq 1 + \frac 12 \sum_{\psi \in \Psi_J} n^{-1/2} \ep_\psi \psi\,.
\end{equation*}
As long as $n^{-1/2} 2^{J(s + d/2)} \lesssim 1$, the functions $f_\ep$ all lie in $\bes{s}{p'}{q}(L; m)$.

Moreover, note that for any $\ep, \ep' \in \{\pm 1\}^{|\Psi_J|}$, \cref{prop:wavelet_lb} implies
\begin{equation*}
W_1({f_\ep}, {f_{\ep'}}) \gtrsim \|f_\ep - f_{\ep'}\|_{\bes{-1}{1}{\infty}} = 2^{-J(1+d/2)} n^{-1/2} \rho(\ep, \ep')\,.
\end{equation*}
where $\rho(\ep, \ep')$ is the Hamming distance between $\ep$ and $\ep'$.

Moreover, when $\rho(\ep, \ep') = 1$, the Hellinger distance satisfies
\begin{equation*}
\int (\sqrt{f_\ep} - \sqrt{f_{\ep'}})^2 \lesssim \int (f_\ep - f_{\ep'})^2 = n^{-1}
\end{equation*}

Therefore, since $W_1$ is a metric, a standard application of Assouad's lemma~\citep[Theorem 2.12]{Tsy09} implies that
\begin{align*}
\adjustlimits \inf_{\tilde \mu} \sup_{f \in \bes{s}{p'}{q}(L; m)} W_1(\hat \mu, {f}) & \gtrsim \adjustlimits  \inf_{\hat \ep} \sup_{\ep \in \{\pm 1\}^{|\Psi_J|}} W_1({f_{\hat \ep}}, {f_\ep})\\
& \gtrsim 2^{-J(1+d/2)} |\Psi_J| n^{-1/2} \\
& \gtrsim 2^{-J} 2^{dJ/2} n^{-1/2}\,,
\end{align*}
where the infimum is taken over all estimators $\tilde \mu$ constructed from $n$ samples and where the final inequality is a consequence of \cref{assume:size}.
Choosing $J$ such that $2^J \asymp n^{\frac{1}{d + 2s}}$ when $d \geq 2$ and $J = 0$ when $d = 1$ yields the claim.
\end{proof}

\section{General smooth densities}\label{sec:unbounded}
In this section, we give proofs for our theorems applying to general smooth densities (\cref{thm:unbounded_lb,thm:unbounded_ub}).
Our main results are matching upper and lower bounds showing showing that the rate of estimation over the class $\besspace$ is strictly worse than the rate over the class $\bounbes$ when $L$ is large enough that $\besspace \not\subseteq \bounbes$.

\subsection{Lower bounds}
We assume that $L$ is large enough that $\bes{s}{p'}{q}(L)$ contains a function $g_0$ whose support lies entirely inside $(0, 1/3)^d$. It is easy to see that this goal is indeed achievable by choosing $g_0$ to be suitable compactly supported smooth bump functions, as long as $L$ is a large enough constant.

The lower bound is based on the following lemma, which gives a lower bound on the Wasserstein distances for a pair of measures with disconnected support.
\begin{lemma}\label{lem:wp_separated_lb}
Let $\mu$ and $\nu$ be measures on $\RR^d$.
Suppose there exist two compact sets $S$ and $T$ such that $d(S, T) \geq c$ and such that the supports of $\mu$ and $\nu$ lie in $S \cup T$.
Then $W_p(\mu, \nu) \geq c |\mu(S) - \nu(S)|^{1/p}$.
\end{lemma}
\begin{proof}
Assume without loss of generality that $\mu(S) \geq \nu(S)$.
Then any coupling between $\mu$ and $\nu$ must assign mass at least $\mu(S) - \nu(S)$ to $S \times T$, so that $W_p^P(\mu, \nu) \geq c^p |\mu(S) - \nu(S)|$.
\end{proof}

The proof of \cref{thm:unbounded_lb} boils down to applying \cref{lem:wp_separated_lb} to appropriately chosen measures.

\begin{proof}[Proof of \cref{thm:unbounded_lb}]
We first prove the $n^{-1/2p}$ bound.
Let $g_0 \in \bes{s}{p'}{q}$ be supported in $[0, 1/3]^d$ and let $g_1$ be a translation of $g_0$ supported on $[2/3, 1]^d$.
For $\lambda \in [-1, 1]$, define $f_\lambda \defeq \frac 12 ((1+\lambda) g_0 + (1-\lambda)g_1)$.
Then for any $\lambda \in [-1, 1]$, the densities~$f_{-\lambda}$ and $f_{\lambda}$ satisfy
\begin{align*}
\int (\sqrt{f_\lambda(x)} - \sqrt{f_{-\lambda}(x)})^2 \dd x & = \frac 12 \int (\sqrt{1+\lambda} - \sqrt{1-\lambda})^2 (g_0(x) + g_1(x)) \dd x \\
& \lesssim \lambda^2\,.
\end{align*}
On the other hand, by \cref{lem:wp_separated_lb}, choosing $S = [0, 1/3]^d$ yields $W_p(f, {f'}) \gtrsim \lambda^{1/p}$.
Therefore, if we choose $\lambda \asymp n^{-1/2}$, then the claim follows from the method of \citet{Lec73}.

We now prove the $n^{-\frac{1 + s/p}{d + s}}$ bound. We proceed via Assouad's lemma. 

Let $g_0$ be as above. For $M > 1$ to be specified, the characterization of $\bes{s}{p'}{q}$ by $\|\cdot\|'_{\bes{s}{p'}{q}}$ implies that there exists a universal constant $c > 0$ such that $h(x) \defeq c M^{-s} g_0(Mx)$ also lies in $\bes{s}{p'}{q}(L)$. We denote by $\Gamma$ a set of vectors in $\RR^d$ such that for any $\gamma_1, \gamma_2 \in \Gamma$, the supports of the functions $h(x - \gamma_1)$, $h(x - \gamma_2)$, and $g_0$ are all separated by at least $c M^{-1}$ for $c$ a small constant. By a volume argument, we can choose $\Gamma$ such that $|\Gamma| \asymp M^d$. We assume that $|\Gamma|$ is even.

We divide the elements of $|\Gamma|$ into pairs and label them $\{(\gamma^+_i, \gamma^-_i)\}_{i = 1}^{|\Gamma|/2}$.
For any $\ep \in \{\pm 1\}^{|\Gamma|/2}$, define
\begin{equation*}
f_\ep \defeq \sum_{i \in |\Gamma|/2}  h(x - \gamma_i^{\ep_i}) + \kappa g_0\,,
\end{equation*}
where $\kappa \defeq 1 - \frac{|\Gamma|}{2} \int h(x) \dd x$ is chosen to ensure that $f_\ep$ integrates to $1$.
Since $\int h(x) \dd x \asymp M^{-s -d}$, the constant $\kappa$ is positive and $f_\ep$ is a density.

Given $\ep, \ep' \in \{\pm 1\}^{|\Gamma|/2}$, define
\begin{equation*}
\Delta(\ep, \ep') \defeq \{i : \ep_i = +1, \ep'_i = -1\}\,.
\end{equation*}
This is the subset of $\Gamma$ present in the density $f_\ep$ but not in $f_{\ep'}$.
We set
\begin{equation*}
S \defeq \bigcup_{i \in \Delta(\ep, \ep')} \supp(h(x - \gamma_i^+))\,.
\end{equation*}
If we denote by $\rho$ the Hamming distance, then the density $f_\ep$ assigns mass $|\Delta(\ep, \ep')| c M^{-s -d} \gtrsim \rho(\ep, \ep') M^{-s -d}$ to $S$, and $f_\ep'$ assigns zero mass to this set.
By construction, the rest of the support of $f_\ep$ and $f_\ep'$ lies at distance at least $c M^{-1}$ from $S$.
Therefore, by \cref{lem:wp_separated_lb}, 
\begin{equation*}
W_p({f_{\ep'}}, {f_\ep}) \gtrsim M^{-1} (\rho(\ep, \ep') M^{-s -d})^{1/p} \gtrsim \rho(\ep, \ep') M^{-\frac s p -1 - d}\,,
\end{equation*}
where the last inequality follows from the fact that $\rho(\ep, \ep') \leq |\Gamma|/2 \lesssim M^d$.
Moreover, if $\rho(\ep, \ep') =1$, then
\begin{equation*}
\int (\sqrt{f_\ep(x)} - \sqrt{f_{\ep'}(x)})^2 \dd x \leq \int |f_\ep(x) - f_{\ep'}(x)| \dd x \lesssim M^{-s -d}\,.
\end{equation*}

Therefore, if we choose $M \asymp n^{\frac{1}{s + d}}$, then Assouad's lemma \citep[Theorem 2.12]{Tsy09} and the fact that $W_p$ satisfies the triangle inequality imply
\begin{align*}
\adjustlimits\inf_{\hat \mu} \sup_{f_\ep} W_p(\hat \mu, {f_\ep}) & \gtrsim \adjustlimits\inf_{\hat \ep} \sup_{\ep} W_p({f_{\hat \ep}}, {f_\ep}) \gtrsim M^{- \frac sp -1} \gtrsim n^{-\frac{1 + s/p}{d +s}}\,.
\end{align*}
as claimed.
\end{proof}
\subsubsection{Upper bounds}
When $p = 1$, there is no longer any dependence on the lower bound $m$ in \cref{thm:besov_wp}, which suggests that the optimal rate of estimation in $W_1$ for general densities on $[0, 1]^d$ should correspond with the one obtained in \cref{thm:estimation_ub} for bounded densities.
And indeed, this is true: the results of \citet{UppSinPoc19} imply that the minimax rates of estimation for densities in $W_1$ is the same irrespective of whether the density is bounded below or not.
However, the duality argument that~\citet{UppSinPoc19} employ breaks down when $p \neq 1$, and the gap between the bounds in \cref{thm:unbounded_lb} and \cref{thm:estimation_ub} hints that the situation when $p \geq 2$ is significantly more subtle.

In this section, we give an estimation matching the rate of \cref{thm:unbounded_lb} rate up to a logarithmic factor when $s \in [0, 1)$ and $f \in \hol{s}{L}$.
Our estimator can only achieves the optimal rate for smoothness $s < 1$ because we rely on particular properties of the Haar wavelet basis.
While this basis is well suited to characterizing $\hol{s}{L}$ when $s < 1$, it cannot exploit further regularity of the unknown density. 
We suspect that a similar result holds for other, more regular choices of wavelet systems, but our techniques do not allow us to prove a more general result.

For $j \geq 0$, denote by $\Psi_j$ the elements of the $d$-dimensional Haar wavelet basis at scale $2^{-j}$~\citep[see][Section 2.3]{Tri10}.
Note that in the case of the Haar wavelet basis over the cube, the set $\Phi$ of scaling functions contains only the function which is identically $1$ on $[0, 1]^d$.

For $j \geq 0$, let $\cQ \defeq \bigcup_{j \geq 0} \cQ_j$ be the dyadic decomposition of $[0, 1]^d$, where $\cQ_j$ consists of a partition of $[0, 1]^d$ into cubes with sides of length $2^{-j}$.
Each element of $\Psi_j$ is supported on a single cube in $\cQ_j$.
We write $V_j$ for the span of $\left\{1 \cup \left(\bigcup_{k < j} \Psi_k\right)\right\}$ and denote by $K_j$ the orthogonal projection onto $V_j$.
The functions in $V_j$ are precisely those which are constant on the elements of $\cQ_j$.

The Haar basis has been used implicitly in existing bounds on the Wasserstein distance in what is known as the ``dyadic partitioning argument''~\citep{WeeBac18,FouGui15,BoiLe-14}.
Suppose that $\mu$ and $\nu$ are two measures on $[0, 1]^d$.
The bound of~\citet[Proposition 1]{WeeBac18} reads
\begin{equation}\label{eq:haar_expansion}
W_p^p(\mu, \nu) \lesssim \sum_{j \geq 0} 2^{-jp} \sum_{Q \in \cQ_j} |\mu(Q) - \nu(Q)|\,.
\end{equation}
When $\mu$ and $\nu$ possess densities $f$ and $g$, respectively, the expression on the right side of the above inequality is an expansion of $f - g$ with respect to the Haar wavelet basis, since $|\mu(Q) - \nu(Q)| = \|K_j (f - g)\|_{L_1(Q)}$.
Written in this way, the bound~\eqref{eq:haar_expansion} therefore implies
\begin{equation}\label{eq:projection_expansion}
W_p(f, g) \lesssim \sum_{j \geq 0} 2^{-j} \|K_j (f - g)\|_{L_1(\Omega)}^{1/p}\,.
\end{equation}

To achieve the optimal rate, we sharpen this bound significantly by observing that, when $f$ is bounded above and below on $Q$, the $L_1$ bound can be replaced by an $L_p$ bound.
We present the refined result as \cref{l1_lp_bound}.

\begin{theorem}\label{l1_lp_bound}
Let $J \geq 0$ be an integer, and let $f$ and $g$ be probability densities satisfying $g/f \in V_J$ (with the convention $0/0 = 0$).
Then
\begin{equation*}
W_p(f, g) \lesssim \sum_{0 \leq j < J} 2^{-j} \left(\sum_{Q \in \cQ_j} \|K_{j+1}(f - g)\|_{L_1(Q)} \wedge m_{Q}^{1-p} \Delta_Q^p \|K_{j+1}(f - g)\|_{L_p(Q)}^p\right)^{1/p}\,.
\end{equation*}
where
\begin{align*}
m_Q \defeq \frac{\int_Q g}{\int_Q f} \cdot \inf_{x \in Q} f(x) \quad \quad \text{and} \quad \quad
\Delta_Q  \defeq \frac{\sup_{x \in Q} f(x)}{\inf_{x \in Q} f(x)}\,.
\end{align*}
\end{theorem}
The quantities $m_Q$ and $\Delta_Q$ control the variation of $f$ on $Q$.
The first is a weighted version of the quantity $m$ appearing in \cref{prop:wavelet_ub}.
The second is small as long as $f$ does not change too much on $Q$, which is the case as long as $f$ is sufficiently smooth.
The proof of \cref{l1_lp_bound} appears in \cref{app:refined}.

In practice, to apply \cref{l1_lp_bound} to bound the Wasserstein distance between arbitrary probability measures $\mu$ and $\nu$, we can find a measure $\bar \nu$ such that $\mu$ and $\bar \nu$ satisfy the assumptions of \cref{l1_lp_bound} and $W_p(\bar \nu, \nu)$ is small.
We adopt this strategy in the proof of the \cref{thm:unbounded_ub}, which appears in \cref{unbounded_ub_proof}.

\section{Computational aspects}\label{sec:comp}

One of the motivations for this line of work is found in applications of optimal transport techniques for data analysis and machine learning, with unknown distributions and access to an independent sample of size $n$.

Many so-called {\em variational Wasserstein problems} involve the problem of minimizing a functional $F:\nu \mapsto W_p(\nu,\mu)$ with unknown $\mu$.
These problems, such as minimum Kantorovich estimators~\citep{BasBodReg06} and Wasserstein barycenters~\citep{AguCar11}, are increasingly common in practical applications~\citep{PeyCut17}, especially when the minimization is taken over a parametric class, with $\nu = \nu_\theta$ for $\theta \in \Theta$.

Solving variational Wasserstein problems in practice requires first obtaining an empirical estimate of the functional $F$ on the basis of data drawn from $\mu$, and then writing the resulting optimization problem in a computationally tractable form. The first issue is typically addressed by obtaining an estimator $\tilde \mu_n$ of $\mu$ and estimating the function via the plug-in principle.
Indeed, the triangle inequality implies that
\[
\sup_{\nu \in \cP}|W_p(\nu,\mu) - W_p(\nu,\tilde \mu_n)| = W_p(\mu, \tilde \mu_n)\, ,
\]
where the supremum is taken over all probability distributions and where equality is achieved at $\mu = \nu$.
Following this approach, guarantees in Wasserstein distance between $\mu$ and the estimator $\tilde \mu_n$ therefore yield uniform deviation bounds for these functionals over the set of probability measures $\cP$ on $\RR^d$.

To solve the resulting optimization problem, finite discretizations are often taken for $\nu$ and $\tilde \mu_n$ to render the resulting problem amenable to discrete optimization techniques~\citep{Cut13,AltWeeRig17}.
For this reason, the estimator $\tilde \mu_n$ is often taken to be the empirical distribution $\hat \mu_n$, since this measure is a finitely supported measure and enjoys the rate
\begin{equation*}
\E W_p(\mu, \hat \mu_n) \lesssim n^{-1/d}\,,
\end{equation*}
as long as $d > 2p$, which is minimax optimal over the class of compactly supported probability measures~\citep{SinPoc18}.

However, in \cref{sec:bounded,sec:unbounded}, we show that under natural regularity assumptions for $\mu$, estimators based on density estimation statistically outperform the empirical distribution. Indeed, focusing on the regime $d > 2p$, \cref{thm:estimation_ub,thm:unbounded_ub} yield guarantees of the form

\[
\E W_p(\mu,\tilde \mu_n) \lesssim n^{-\gamma^*(s)/d}\, ,
\]
where $\gamma^*(s) \geq 1$ increases as the smoothness of $\mu$ increases.

These results are summarized in the following table, highlighting that the optimal exponent $\gamma^*(s)/d$ interpolates between $1/d$ (for $s=0$) a dimension-free rate (for $s$ going to $\infty$). The value $s<1$ is of special interest, as it corresponds to the smoothness which can be exploited by a histogram estimator, which is most relevant for computational aspects. We take this point up further below.

\begin{center}
\begin{tabular}{|c|c|c|c|c|}
\hline
\rule{0pt}{3ex} Nonparametric class & optimal $\gamma^*(s)$ & $\gamma^*(0)/d$  & $\gamma^*(1)/d$& $\gamma^*(\infty)/d$\\
\hline
\rule{0pt}{3ex} $\bounbes$ & $\frac{1+s}{1+2s/d}$ & $1/d$ &  $\frac{2}{d+2}$ & $1/2$\\ [1ex]
\rule{0pt}{3ex} $\besspace,  \,\, s < 1$ & $\frac{1+s/p}{1+s/d}$ & $1/d$ & $\frac{1+1/p}{d+2}$  & ---\\[1ex]
\hline
\end{tabular}
\end{center}

Note that the exponent $\gamma^*(s)$ is greater than 1 for all $s > 0$. In the light of these results, there is an apparent tension between two objectives: statistical precision and computational efficiency. On the one hand, using atomic mesures as estimators of unknown distributions allows to efficiently compute Wasserstein distances: the optimal transport problem reduces to a linear program in finite dimension. Their statistical performance can however be suboptimal, as shown in the table above: for smooth densities the lower bound in $n^{-1/d}$ applies to all $n$-atomic distributions~\citep{Dud69}. On the other hand, estimators with smooth densities have minimax optimal statistical precisions, but are defined through a projection in Besov norm of a wavelet-based density estimator. Even in the simple case of histograms (piecewise-constant densities), there is no explicit or simple way to solve optimal transport problems.

We therefore propose a procedure to leverage the regularity of the distribution, and to handle our proposed estimators in a computationally efficient manner. The idea is to exploit the best of both worlds, by creating an atomic measure with $M >n$ atoms, based on a density estimator. They statistically outperform the empirical measure, and optimal transport problems can be explicitly solved on these measures. If it is possible to efficiently sample from one of these estimators, it is always possible to extract an atomic distribution out of it, as described in the following
\begin{defin}
Let $\tilde \mu_n$ be a probability measure from which one can draw sample points, and let $Z_1,\ldots,Z_M$ be an i.i.d. sample from $\tilde \mu_n$. We denote by $\bar \mu_{n,M}$ the {\em estimator resample distribution}, the empirical distribution of the $Z_i$s.
\end{defin}
The distribution $\bar \mu_{n,M}$ is ``simply'' the empirical distribution of a sample of size $M$ from a distribution. However, we retain $n$ in the notation, to highlight that the $Z_i$s are themselves drawn from an estimator based on a sample of size $n$ from an unknown $\mu$. We recall the following result for compactly supported distributinos.
\begin{prop}[\citealp{FouGui15}]
For $d > 2p$, the estimator resample distribution $\bar \mu_{n,M}$ satisfies
$
\E W_p(\tilde \mu_n, \bar \mu_{n,M}) \lesssim M^{-1/d}.
$
\end{prop}
As a consequence of this result, resampling from the estimated distribution yields an atomic measure as close in Wasserstein distance to the original estimator as desired, since $M$ can be chosen by the statistician. 
As the following makes clear, by choosing $M > n$, it is possible to obtain atomic distributions with estimation error as good as the original estimator $\tilde \mu_n$.
\begin{cor}\label{resample}
Assume $d > 2p$.
Let $\mu$ be in a nonparametric class such that there exists an estimator $\tilde \mu_n$ from which one can efficiently sample, and such that
\[
\E W_p(\tilde \mu_n, \mu) \lesssim n^{-\gamma^*/d}
\]

For any $\gamma \in [1,\gamma^*]$, the estimator resample distribution $\bar \mu_{n,M}$ with $M=n^\gamma$ satisfies
\[
\E W_p(\mu, \bar \mu_{n,M}) \lesssim n^{-\gamma/d}\, .
\]
\end{cor}
\begin{proof}
We have by the triangle inequality that
\begin{align*}
\E W_p(\mu, \bar \mu_{n,M}) &\le \E W_p(\mu,\tilde \mu_n) + \E W_p(\tilde \mu_n, \bar \mu_{n,M}) \lesssim n^{-\gamma^*/d} + M^{-1/d} \lesssim n^{-\gamma/d}\, ,
\end{align*}

having taken $M$ of order $n^\gamma$.
\end{proof}

It is now straightforward to give a proof of \cref{thm:ub_comp}.
\begin{proof}[Proof of \cref{thm:ub_comp}]
We consider the estimator resample distribution constructed from the histogram estimator give by \cref{thm:unbounded_ub}.
If $d \leq 2p$, then \citet[Theorem 1]{FouGui15} implies that the empirical distribution $\hat \mu_n$ already achieves the rate appearing in \cref{thm:unbounded_ub} up to a logarithmic factor; therefore, choosing $M = n$ suffices.
If $d > 2p$, we apply \cref{resample}. \Cref{thm:unbounded_ub} implies that, up to logarithmic factors, the histogram estimator $\mu_{\hat f}$ achieves the rate $n^{-\gamma^*/d}$ with $\gamma^* = \frac{1 + \frac{s}{p}}{1 + \frac{s}{d}} \wedge \frac{d}{2p}$.
Choosing $\gamma = \gamma^*$ and noting that $\gamma^* < 2$ yields the claim.

Finally, to show that constructing $\bar \mu_{n, M}$ takes time $O(M)$, we note that the histogram estimator constructed in \cref{thm:unbounded_ub} is piecewise constant with $O(n)$ pieces. Constructing the histogram estimator takes $O(n)$ time, and sampling $M$ points from a histogram distribution can be done in time $O(n + M) = O(M)$ by the alias method~\citep{KroPet79}.
\end{proof}

In some examples, the choice of $\gamma^*$ can be left to the practitioner. For our estimators, this corresponds to choosing the depth of the wavelet decomposition. Taking piecewise constant estimators (histograms) limits the exponent $\gamma$ to $\gamma^*=\gamma^*(1)$. In any case, it is also possible to chose $M=n^{\gamma}$ for $\gamma \in (1,\gamma^*(1)]$, and let the approximation error $n^{-\gamma/d}$ dominate the statistical error $n^{-\gamma^*(1)/d}$.

Using the estimator resample distribution $\bar \mu_{n,M}$ instead of $\hat \mu_n$ requires solving optimal transport problems of size $M=n^\gamma$ instead of $n$.
This naturally increases the computational cost. This motivates the question of quantifying the statistical and computational tradeoffs of our proposal. The dependency of the algorithmic cost of solving optimal transportation problems on the size of the distribution is the subject of a large literature~\citep[see][]{PeyCut17}, from which we extract a simple bound.

\begin{prop}[\citealp{AltWeeRig17,DvuGasKro18}]\label{prop:sink}
Given two distributions $\alpha$ and $\beta$ supported on at most $M$ atoms on a set of diameter $1$, an additive $\ep$ approximation to $W^p_p(\alpha, \beta)$ can be computed in time $O(M^2 \log(M) / \varepsilon^2)$. 
 \end{prop}

The following describes the interplay between statistical precision and computational efficiency for 
\begin{theorem}
Let $\mu$ be in a nonparametric class such that there exists an estimator $\tilde \mu_n$ from which one can efficiently sample, and such that
\[
\E W_p(\tilde \mu_n, \mu) \lesssim n^{-\gamma^*/d}
\]
Given a sample of size $n$ from $\mu$ and known $\nu$, for any $\gamma \in [1,\gamma^*]$ an estimate $\tilde W_{p,n}$ of $W_p(\nu,\mu)$ satisfying
\[
\E |\tilde W_{p,n} - W_p(\nu,\mu)| \lesssim n^{-\gamma/d}
\]
can be computed in time $O\big(n^{\gamma(2+2/d)} \log(n)\big)$.
\end{theorem}
\begin{proof}
Taking $\tilde \mu_n$ such that $\E W_p(\mu,\tilde \mu_n) \le n^{-\gamma^*/d}$, $\bar \mu_{n,M}$ the empirical resample distribution of $\tilde \mu_n$, and $\bar \nu_M$ an $M$-atomic version of $\nu$ (obtained by sampling from or discretizing $\nu$), we have
\begin{align*}
\E |W_p(\bar \nu_M, \bar \mu_{n,M}) - W_p(\nu,\mu)| &\le \E W_p(\mu, \bar \mu_{n,M}) + \E W_p(\nu, \bar \nu_{M})\\
&\lesssim n^{-\gamma^*/d} + M^{-1/d} \lesssim n^{-\gamma/d} \, ,
\end{align*}
by taking $M=n^{\gamma}$. Taking $\varepsilon = n^{-\gamma/d}$, computing an $\varepsilon$-approximation $\tilde W_{p,n}^p$ to $W^p_p(\bar \nu_M, \bar \mu_{n,M})$ given $\bar \mu_{n,M}$ and $\bar \nu_M$ takes time $O\big(n^{\gamma(2+2/d)} \log(n)\big)$ by \cref{prop:sink}.
Finally, the inequality
$$|\tilde W_{p, n} - W_p(\bar \nu_M, \bar \mu_{n,M})| \leq |\tilde W^p_{p, n} - W^p_p(\bar \nu_M, \bar \mu_{n,M})| \leq \varepsilon = n^{-\gamma/d}$$
yields the claim.
\end{proof}

Taking $\gamma=1$, $\tilde \mu_n = \hat \mu_n$ and $M=n$ with $\bar \mu_{n,M} = \hat \mu_n$ (without resampling) is always possible. It yields an algorithm that outputs a $n^{-1/d}$ approximation in time $O\big(n^{2+2/d} \log(n)\big)$. However, whenever another estimator $\tilde \mu_n$ with precision $n^{-\gamma/d}$ exists for $\gamma \in (1,\gamma^*(s)]$, it is possible to obtain a better approximation with error $\varepsilon_n(\gamma) = n^{-\gamma/d}$ in time of order $T_n(\gamma)=n^{\gamma(2+2/d)}$, up to terms logarithmic in $n$. This quantifies the computational cost for added statistical precision. 
We summarize these results in the following table, for $\gamma^*= \gamma^*(1)$ for histogram estimators (from which it is easy to sample).
The parameter $\gamma$ can be taken in the full range from $1$ to $\gamma^*(1)$.
\begin{center}
\begin{tabular}{|c|c|c|c|}
\hline
\rule{0pt}{3ex} $\gamma$ & precision  $n^{-\gamma/d}$& $M = n^\gamma$ & time $n^{\gamma(2+2/d)} \log(n)$ \\
\hline
\rule{0pt}{3ex} $\gamma=1$ & $n^{-\frac{1}{d}}$ & $n$ & $n^{2+2/d} \log(n)$ \\ [1ex]
\rule{0pt}{3ex} $\gamma^*(1) = \frac{1+1/p}{1+1/d}$ & $n^{-\frac{1+1/p}{d+1}}$ & $n^{\frac{1+1/p}{1+1/d}}$ & $n^{2+2/p} \log(n)$  \\[1ex]
\hline
\end{tabular}
\end{center}

In the high-dimensional limit, 
we obtain that a histogram estimator can improve the exponent in the precision by a factor $\gamma$ of nearly 2 at the price of increasing the exponent in the running time by nearly the same factor.
The choice of $M$, which can be left to the statistician, determines the value of $\gamma$.
 
\appendix
\numberwithin{equation}{section}
\numberwithin{lemma}{section}
\numberwithin{assume}{section}
\section{Omitted proofs}\label{app:proofs}
\subsection{Proof of \cref{interp_lb}}
The bound is invariant under the reparametrization $t \mapsto 1 - t$, so it suffices to prove the claim for $t \leq 1/2$.
By definition, we have
\begin{equation*}
\rho(x, t) = (1 - \lambda(t))f(x) + \lambda(t) g(x) \geq  \lambda(t) (f(x) + g(x))\,,
\end{equation*}
where we have used the fact that $1 - \lambda(t) \geq \lambda(t)$ when $t \leq 1/2$.
Since $f(x) + g(x) \geq f(x) \vee g(x) \geq m$, the claim follows.
\qed

\subsection{Proof of \cref{lem:vphi_divergence}}
The partial derivatives of this vector field satisfy
\begin{align*}
\frac{\mathrm{d}}{\mathrm{d}x_1} (V_\phi)_1 & = \phi^{(1)}(x^{(1)}) \\
\frac{\mathrm{d}}{\mathrm{d}x_i} (V_\phi)_i & = \phi^{(i)}(x^{(i)}) -  \phi^{(i-1)}(x^{(i-1)}) \quad \quad 1 < i \leq d\,,
\end{align*}
so that $\nabla \cdot V = \phi$ on the interior of $[0, 1]^d$.

On the boundary $x_i = 0$ we have $(V_\phi)_i = 0$ for all $i$, and for $i > 1$ on the boundary $x_i = 1$ we have
\begin{equation*}
(V_\phi)_i = \int_0^1 \phi^{(i)}(x^{(i-1)}, t_i) \dd t_i - \phi^{(i-1)}(x^{(i-1)}) = 0\,.
\end{equation*}
Finally, on $x_1 = 1$, we have $(V_\phi)_1 = \int_0^1 \phi^{(1)}(t_1) \dd t_1 = \int_{[0, 1]^d} \phi(t) \dd t$.
Since $f$ and $g$ are both probability densities and the constant function $1$ is in the span of $\Phi$, we have $\int_{[0, 1]^d} f(x) \dd x = \int_{[0, 1]^d} \sum_{\phi \in \Phi} \alpha_\phi \phi(x) \dd x = 1$, and, analogously, $\int_{[0, 1]^d} \sum_{\phi \in \Phi} \alpha'_\phi \phi(x) \dd x = 1$.
We therefore obtain that $\sum_{\phi \in \Phi} (\alpha_\phi - \alpha'_\phi) (V_\phi)_1 = 0$ on the face $x_1 = 0$; thus, the boundary conditions hold as well.
\qed
\subsection{Proof of \cref{lem:vpsi_divergence}}
By construction, $\nabla \cdot V_\psi = \frac{d (V_\psi)_k}{d x_k} = \psi$ on the interior of $[0, 1]^d$.
The equality $(V_\phi)_i  = 0$ holds on the boundaries $x_i = 0$ and $x_i =1$ for $i \neq k$, and by construction $(V_\phi)_k = 0$ on $x_k = 0$.
Finally, since $\int_0^{1} \psi_k(t) \dd t = 0$, we also have $(V_\phi)_k = 0$ on $x_1 = 1$, and thus the boundary conditions are satisfied.\qed

\subsection{Proof of \cref{lem:vphi_norm_estimate}}
We will show that $\left\|\sum_{\phi \in \Phi} \alpha_\phi V_\phi\right\|_{L_p([0, 1]^d)} \lesssim \left\|\sum_{\phi \in \Phi} \alpha_\phi \phi\right\|_{L_p([0, 1]^d)}$ and conclude by appealing to the stability property of the scaling functions.

The definition of $V_\phi$ implies that
\begin{align*}
\left|\sum_{\phi \in \Phi} \alpha_\phi (V_\phi)_i(x) \right|^p & = \left| \int_0^{x_i} \sum_{\phi \in \Phi} \alpha_\phi \phi^{(i)}(x^{(i-1)}, t_i) \dd t_i - x_i \sum_{\phi \in \Phi} \alpha_\phi \phi^{(i-1)}(x^{(i-1)})\right|^p \\
& \leq 2^{p-1} \left| \int_0^{x_i} \sum_{\phi \in \Phi} \alpha_\phi \phi^{(i)}(x^{(i-1)}, t_i) \dd t_i\right|^p + 2^{p-1} x_i \left|\sum_{\phi \in \Phi} \alpha_\phi \phi^{(i-1)}(x^{(i-1)})\right|^p \\
& \lesssim \int_0^1\dots \int_0^1 \left|\sum_{\phi \in \Phi} \alpha_\phi \phi(x^{(i-1)}, t_i, \dots, t_d)\right|^p \dd t_{i} \cdots \dd t_d\,,
\end{align*}
where in the second inequality we use Jensen's inequality and the fact that $x_i \leq 1$.

We obtain
\begin{align*}
\left\|\sum_{\phi \in \Phi} \alpha_\phi V_\phi\right\|_{L_p}^p  \lesssim \max_{i \in [d]} \int_\Omega \left|\sum_{\phi \in \Phi} \alpha_\phi (V_\phi)_i \right|^p
 \lesssim \int_{\Omega} \left|\sum_{\phi \in \Phi} \alpha_\phi \phi\right|^p
 = \left\|\sum_{\phi \in \Phi} \alpha_\phi \phi\right\|^p_{L_p}
\end{align*}

The claim then follows from \cref{assume:stability}.\qed
\subsection{Proof of \cref{lem:vpsi_norm_estimate}}
By \cref{assume:locality}, there exists an interval $I$ with $|I| \lesssim 2^{-j}$ such that $\supp(\psi_i) \subseteq I$.
Since $(V_\psi)_i(x_i) = 0$ if $x \notin I$, H\"older's inequality implies that
\begin{align*}
\left| \int_0^{x_i} \psi_i(t) \dd t\right|^p & = \1\{x_i \in I\} \left| \int_0^{x_i} \psi_i(t) \1\{t \in I\} \dd t\right|^p \\
& \leq |I|^{p-1} \1\{x_i \in I\} \int_0^{x_i} |\psi_i(t)|^p \dd t \\
& \lesssim 2^{-j(p-1)} \1\{x_i \in I\} \int_0^1 |\psi_i(t)|^p \dd t\,.
\end{align*}
We therefore obtain by \cref{assume:norm} that
\begin{equation*}
\int_{[0, 1]^d} \|V_\psi\|^p \lesssim 2^{-j(p-1)} 2^{-j} \int_{[0, 1]^d} |\psi|^p \lesssim 2^{-jp} 2^{pdj(\frac 12 - \frac 1p)}\,.
\end{equation*}

The construction of $V_\psi$ and \cref{assume:locality} imply that $V_\psi(x) = 0$ if $x \notin I_\psi$, where $\left\|\sum_{\psi \in \Psi_j} \1\{x \in I_\psi\}\right\|_{L_\infty} \lesssim 1$.
H\"older's inequality therefore yields
\begin{align*}
\left\|\sum_{\psi \in \Psi_j} \beta_\psi V_\psi\right\|^p_{L_p} & \leq \int_{[0, 1]^d} \left(\sum_{\psi \in \Psi_j} |\beta_\psi| \|V_\psi\| \1\{x \in I_\psi\} \right)^p \dd x \\
&  \leq \sum_{\psi \in \Psi_j} |\beta_\psi|^p \int_{[0, 1]^d} \|V_\psi\|^p \left(\int_{[0, 1]^d}\sum_{\psi \in \Psi_j} \1\{x \in I_\psi\}\dd x\right)^{p-1} \\
& \lesssim 2^{-jp} 2^{pdj(\frac 12 - \frac 1p)} \|\beta_j\|_{\ell_p}^p\,.
\end{align*}\qed

\subsection{Proof of \cref{lem:h_estimate}}
This follows directly from the assumptions on the multiresolution analysis:
\begin{align*}
\|\nabla h\|_{L_{p'}(\Omega)} & \leq \left\|\nabla \sum_{\phi \in \Phi} \kappa_\phi \phi\right\|_{L_{p'}(\Omega)} + \left\|\nabla \sum_{\psi \in \Psi_j} \lambda_\psi \psi  \right\|_{L_{p'}(\Omega)} \\
& \lesssim \left\|\sum_{\phi \in \Phi} \kappa_\phi \phi\right\|_{L_{p'}(\Omega)} + 2^j\left\|\sum_{\psi \in \Psi_j} \lambda_\psi \psi  \right\|_{L_{p'}(\Omega)} \\
& \lesssim \|\kappa\|_{\ell_{p'}} + 2^{j + dj(\frac 12 - \frac{1}{p'})} \|\lambda\|_{\ell_{p'}} \\
& \leq 1 + 2^{dj(1 - \frac 1p - \frac{1}{p'})} = 2\,,
\end{align*}
where the inequalities follow, respectively, from the triangle inequality, \cref{assume:bernstein}, \cref{assume:stability}, and the assumption on $\kappa$ and $\lambda$.
\qed

\subsection{Proof of \cref{prop:moment_bounds}}
We first show the claim for $\|\beta_j - \tilde \beta_j\|_{\ell_p}$, $j \geq 0$.
The inequalities of \citet{Ros72} imply that there exists a constant $c_p$ such that for any $\psi \in \Psi_j$,
\begin{align*}
\E |\beta_\psi - \tilde \beta_\psi|^p \leq c_p \left(\frac{\sigma_\psi^p}{n^{p/2}} + \frac{\E \left|\psi(X) - \E \psi(X)\right|^p}{n^{p-1}}\1\{p \geq 2\}\right)\,,
\end{align*}
where $\sigma_\psi^2 \defeq \E \left|\psi(X) - \E \psi(X)\right|^2$ and $X \sim f$.

\Cref{assume:norm} implies that $\|\psi\|_{L_\infty} \lesssim 2^{dj/2}$, so
\begin{equation*}
\frac{\E \left|\psi(X) - \E \psi(X)\right|^p}{n^{p-1}} \lesssim \frac{2^{dj(p-2)/2} \sigma_\psi^2}{n^{p-1}} \leq \frac{\sigma_\psi^2}{n^{p/2}}\,,
\end{equation*}
where the last inequality follows from the assumption that $n \geq 2^{dj}$.

In order to establish that $\E \|\beta_j - \hat \beta_j\|_{\ell_p}^p \leq \frac{2^{dj}}{n^{p/2}}$, it therefore suffices to show that
\begin{equation*}
\sum_{\psi \in \Psi_j} (\E \left|\psi(X) - \E \psi(X)\right|^2)^{k/2} \leq 2^{dj}\,,
\end{equation*}
for $k \in \{2, p\}$.

We have
\begin{equation*}
(\E \left|\psi(X) - \E \psi(X)\right|^2)^{k/2} \leq (\E \psi(X)^2)^{k/2} = \left(\int_\Omega \psi(x)^2 f(x) \dd x\right)^{k/2} \leq \int_\Omega \psi(x)^2 f(x)^{k/2} \dd x\,,
\end{equation*}
where in the last step we use Jensen's inequality combined with the fact that $\int \psi^2 = 1$ for all $\psi \in \Psi_j$ (see \cref{assume:basis}).
Finally, we obtain
\begin{align*}
\sum_{\psi \in \Psi_j} (\E \left|\psi(X) - \E \psi(X)\right|^2)^{k/2} & \leq \int_\Omega \sum_{\psi \in \Psi_j} \psi(x)^2 f(x)^{k/2} \dd x \\
& \leq \left\|\sum_{\psi \in \Psi_j} \psi(x)^2 \right\|_{L_\infty} \|f\|^{k/2}_{L_{k/2}}
\end{align*}
The fact that $f \in \bes{s}{p'}{q}(L; m)$ where $1 \leq p \leq p'$ implies that $\|f\|^{k/2}_{L_{k/2}} \leq \|f\|^{p'}_{L_{p'}} \lesssim 1$, and \cref{assume:locality,assume:norm} imply $\left\|\sum_{\psi \in \Psi_j} \psi(x)^2 \right\|_{L_\infty} \lesssim 2^{dj}$. This proves the claim.

To show the analogous bound on $\|\alpha - \tilde \alpha\|_{\ell_p}$, we again write
\begin{equation*}
\E |\alpha_\phi - \tilde \alpha_\phi|^p \leq c_p\left(\frac{\sigma_\phi^p}{n^{p/2}} + \frac{\E \left|\phi(X) - \E \phi(X)\right|^p}{n^{p-1}}\1\{p \geq 2\}\right)
\end{equation*}
for any $\phi \in \Phi$, where $\sigma^2 \defeq \left|\phi(X) - \E \phi(X)\right|^2$.
Since $\|\phi\|_{L_\infty} \lesssim 1$ for any $\phi \in \Phi$, this bound simplifies to
\begin{equation*}
\E |\alpha_\phi - \tilde \alpha_\phi|^p \lesssim \frac{1}{n^{p/2}}\,,
\end{equation*}
and hence
\begin{equation*}
\E \|\alpha - \hat \alpha\|_{\ell_p}^p \leq  \frac{\left|\Phi\right|}{n^{p/2}} \lesssim \frac{1}{n^{p/2}}
\end{equation*}
by \cref{assume:size}.
\qed

\section{Adaptivity bound}\label{adaptivity_bound}
Let us first justify the definition of $\tau_j$, by showing that it gives an upper bound on $\|\tilde \beta_j - \beta_j\|_{\ell_p}$ with high probability.
\begin{lemma}\label{bernstein_bound}
Suppose $f \in \bes{s}{p'}{q}(L)$ and $\tau_j$ is defined as in \eqref{tau_j}.
For any $j \geq 0$, if $n \geq 2^{dj}$, it holds
\begin{equation}\label{tau_tail_bound}
\p[ \|\tilde \beta_j - \beta_j\|_{\ell_p} \geq c \tau_j] \lesssim e^{-2cdj} \quad \forall{c \geq 1/2}\,.
\end{equation}
Moreover,
\begin{equation}\label{tau_expectation_bound}
\E \|\tilde \beta_j - \beta_j\|_{\ell_p}^2 \lesssim \tau_j^2\,.
\end{equation}
\end{lemma}
\begin{proof}
For any $\psi \in \Psi_j$, Bernstein's inequality~\citep{BouLugMas13} yields
\begin{equation*}
\p[|\tilde \beta_\psi - \beta_\psi| \geq t) \leq 2 \exp(- nt^2/(2 \sigma_\psi^2 + 2 \|\psi\|_{L_\infty} t/3))\,,
\end{equation*}
where $\sigma_\psi^2$ denotes the variance of $\psi(X)$ when $X \sim f$.
In particular, this yields that
\begin{equation*}
\p\left[|\tilde \beta_\psi - \beta_\psi| \geq 2 \frac{\sigma_\psi \sqrt t}{\sqrt n} + 2 \frac{\|\psi\|_{L_\infty} t}{n}\right] \leq 2 e^{-t}\,.
\end{equation*}
Taking a union bound over all $\psi \in \Psi_j$ yields
\begin{equation*}
\p \left[ \|\tilde \beta_j - \beta_j\|_{\ell_p} \geq 2 \left(\sum_{\psi \in \Psi_j} \sigma_\psi^p\right)^{1/p}\frac{\sqrt t}{\sqrt n} + 2 \left(\sum_{\psi \in \Psi_j} \|\psi\|_{L_\infty}^p\right)^{1/p} \frac t n\right] \leq 2 |\Psi_j| e^{-t}\,.
\end{equation*}
As in the proof of \cref{prop:moment_bounds}, there exists a constant $K_0$ depending only on the wavelet basis such that $\sum_{\psi \in \Psi_j} \sigma_\psi^p \leq K_0^p L^p 2^{dj}$.
By \cref{assume:norm,assume:size}, there exists a constant $K_1$ depending on the wavelet basis such that $\left(\sum_{\psi \in \Psi_j} \|\psi\|_{L_\infty}^p\right)^{1/p} \leq K_1 2^{dj (\frac 12 + \frac 1p)}$.
Likewise, \cref{assume:size} implies the existence of a constant $K_2$ such that $|\Psi_j| \leq K_2 2^{dj}$.

Therefore, setting $K = \max\{K_0, K_1, K_2\}$ and using the assumption that $n \geq 2^{dj}$ implies
\begin{equation*}
\p\left[ \|\tilde \beta_j - \beta_j\|_{\ell_p} \geq 2 K  2^{dj/p} \frac{L \sqrt t + t}{\sqrt n}\right] \leq 2 K 2^{dj} e^{-t}\,.
\end{equation*}
Choosing $t = 4 cdj$ and applying the definition of $\tau_j$ yields~\eqref{tau_tail_bound}.

The bound~\eqref{tau_expectation_bound} follows from integrating the estimate~\eqref{tau_tail_bound}.
\end{proof}

We now turn to the proof of \cref{thm:estimation_ub_adaptive}.
By an argument analogous to the one given in \cref{lem:wp_to_bes}, it suffices to bound
\begin{equation*}
\E \|f - \tilde f\|_{\cB^{-1}_{p, q}} = \E \sum_{j \geq 0} 2^{-j} 2^{dj(\frac 12 - \frac 1p)} \left(\|\tilde \beta_j - \beta_j\|_{\ell_p} \indic{\|\tilde \beta_j\|_{\ell_p} \geq \tau_j} + \|\beta_j\|_{\ell_p} \indic{\|\tilde \beta_j\|_{\ell_p} < \tau_j}\right)
\end{equation*}

We require two simple lemmas bounding the expectations of each summand.
\begin{lemma}\label{trunc}
For any $j \geq 0$,
\begin{equation*}
\E \|\beta_j\|_{\ell_p} \indic{\|\tilde \beta_j\|_{\ell_p} < \tau_j} \lesssim \tau_j \wedge \|\beta_j\|_{\ell_p}\,.
\end{equation*}
\end{lemma}
\begin{proof}
We compute
\begin{align*}
\E \|\beta_j\|_{\ell_p} \indic{\|\tilde \beta_j\|_{\ell_p} < \tau_j} & \leq \E \|\beta_j - \tilde \beta_j\|_{\ell_p} + \E \|\tilde \beta_j\|_{\ell_p} \indic{\|\tilde \beta_j\|_{\ell_p} < \tau_j} \\
& \leq \E \|\beta_j - \tilde \beta_j\|_{\ell_p} + \tau_j \\
& \lesssim \tau_j\,,
\end{align*}
where the last inequality follows from \cref{bernstein_bound},~\eqref{tau_expectation_bound}.
Combining this estimate with the trivial bound $\E \|\beta_j\|_{\ell_p} \indic{\|\tilde \beta_j\|_{\ell_p} < \tau_j} \leq \|\beta_j\|_{\ell_p}$ gives the claim.

\end{proof}

\begin{lemma}\label{deviation}
For any $j \geq 0$,
\begin{equation*}
\E \|\tilde \beta_j - \beta_j\|_{\ell_p} \indic{\|\tilde \beta_j\|_{\ell_p} \geq \tau_j} \lesssim \left\{
\begin{array}{ll}
2^{-dj/2} \tau_j & \text{if $\|\beta_j\|_{\ell_p} < \frac 12 \tau_j$.} \\
\tau_j & \text{otherwise}
\end{array}\right.
\end{equation*}
\end{lemma}
\begin{proof}
If $\|\beta_j\|_{\ell_p} < \frac 12 \tau_j$, we use the Cauchy-Schwarz inequality and \cref{bernstein_bound} to obtain
\begin{align*}
\E \|\tilde \beta_j - \beta_j\|_{\ell_p} \indic{\|\tilde \beta_j\|_{\ell_p} \geq \tau_j} & \leq \E \|\tilde \beta_j - \beta_j\|_{\ell_p} \indic{\|\tilde \beta_j - \beta_j\|_{\ell_p} \geq \frac 12 \tau_j} \\
& \leq \p\Big[\|\tilde \beta_j - \beta_j\|_{\ell_p} \geq \frac 12 \tau_j\Big]^{1/2} \left(\E \|\tilde \beta_j - \beta_j\|_{\ell_p}^2\right)^{1/2}  \\
& \lesssim 2^{-dj/2} \tau_j \,.
\end{align*}

For the second bound, we simply use use $\E \|\tilde \beta_j - \beta_j\|_{\ell_p} \indic{\|\tilde \beta_j\|_{\ell_p} < \tau_j} \leq \E \|\tilde \beta_j - \beta_j\|_{\ell_p} \lesssim \tau_j$.
\end{proof}

We are now ready to complete the proof.
\begin{proof}[Proof of \cref{thm:estimation_ub_adaptive}]
By \cref{trunc,deviation}, we have the bound
\begin{equation*}
\E \|f - \tilde f\|_{\cB^{-1}_{p, q}} \lesssim \sum_{j \geq 0} 2^{-j} 2^{dj(\frac 12 - \frac 1p)} \left\{(\tau_j \wedge \|\beta_j\|_{\ell_p}) + 2^{-dj/2} \tau_j\right\}\,.
\end{equation*}
First, the definition of $\tau_j$ implies that
\begin{equation*}
2^{-j} 2^{dj(\frac 12 - \frac 1p)} \cdot 2^{-dj/2} \tau_j \lesssim 2^{-j} \frac{dj}{\sqrt n}\,,
\end{equation*}
so $\sum_{j \geq 0} 2^{-j} 2^{dj(\frac 12 - \frac 1p)} \tau_j p_j^{1/2} \lesssim n^{-1/2}$.
Since the claimed rate is no faster than $n^{-1/2}$, this term is not the dominant one.

If $f \in \bes{s}{p'}{q}(L; m)$, then $2^{dj(\frac 12 - \frac 1p)} \|\beta_j\|_{\ell_p} \lesssim 2^{-js}$, which implies
\begin{equation*}
2^{-j} 2^{dj(\frac 12 - \frac 1p)} (\tau_j \wedge \|\beta_j\|_{\ell_p}) \lesssim 2^{-j} j \left(\frac{2^{dj}}{n}\right)^{1/2} \wedge 2^{-j(1+s)}\,,
\end{equation*}
Summing this expression yields
\begin{equation*}
\sum_{j \geq 0} 2^{-j} 2^{dj(\frac 12 - \frac 1p)} (\tau_j \wedge \|\beta_j\|_{\ell_p}) \lesssim \left\{
\begin{array}{ll}
n^{-1/2} & \text{ if $d = 1$} \\
n^{-1/2} (\log n)^2 & \text{ if $d = 2$} \\
n^{-(1+s)/(d+2s)} \log n & \text{ if $d \geq 3$.}
\end{array}
\right.
\end{equation*}
\end{proof}

\section{An upper bound for densities not bounded below}\label{unbounded_ub_proof}
We require two basic facts about Binomial random variables; the first is a direct consequence of the inequalities of \citet{Ros72} and the second is a standard Chernoff bound~\citep{HagRub90}.
\begin{lemma}\label{binomial_moments}
Let $X \sim \mathrm{Bin}(n, p)$. If $n p \geq 1$, then
\begin{equation*}
\E \left| \frac{X}{n} - p\right|^k \lesssim (p/n)^{k/2}\,. 
\end{equation*}\textbf{}
\end{lemma}

\begin{lemma}\label{binomial_deviations}
If $X \sim \mathrm{Bin}(n, p)$, then
\begin{equation*}
\p[ X \leq \frac 12 np] \leq e^{-\frac{np}{8}}\,.
\end{equation*}
\end{lemma}

Throughout the proof, we let $\Psi_j$ for $j \geq 0$ denote the Haar wavelet basis.
Fix a threshold $J$, which we will assume satisfies $n \asymp 2^{J(d+s)}$.
We set
\begin{equation*}
\hat \beta_\psi \defeq \frac 1n \sum_{i=1}^n \psi(X_i) \quad \psi \in \Psi_j, 0 \leq j \leq J
\end{equation*}
and define
\begin{equation*}
\hat f \defeq 1 + \sum_{0 \leq j < J}\sum_{\psi \in \Psi_j} \hat \beta_\psi \psi\,
\end{equation*}

For convenience, throughout the proof we write $\mu$ for the probability measure with density $f$ and $\nu$ for the probability measure with density $\hat f$.
Note that for any $Q \in \cQ_j$, $j \leq J$, we have $n \cdot \nu(Q)\sim\mathrm{Bin}(n, \mu(Q))$.

Let us now define a density $g$ by
\begin{equation*}
g(x) = \frac{\nu(Q)}{\mu(Q)} f(x) \quad \forall x \in Q, Q \in \cQ_J\,,
\end{equation*}
where we interpret $0 / 0  = 0$.
Since $\mu(Q) = 0 \implies \nu(Q) = 0$ almost surely, this density is well defined with probability $1$.
By construction, $g/f$ is constant on each element of $\cQ_J$, so $g/f  \in V_J$.

Let $0 \leq j < J$, and fix $Q \in \cQ_j$.
\Cref{l1_lp_bound} implies that the contribution to $W_p(f, g)$ from $Q \in \cQ_j$ can be bounded by either the $L_1$ bound $\|K_{j+1}(f - g)\|_{L_1(Q)}$ or the $L_p$ bound $m_Q^{1-p} \Delta_Q^p \|K_{j+1}(f - g)\|^p_{L_p(Q)}$.
We can always apply the first bound, but we will apply the second when the density $f$ is bounded below on $Q$.
We denote by $\cE_Q$ the high-probability event that $\nu(Q) \geq \frac 12 \mu(Q)$.

We divide into cases.
As we shall see, the dominant term arises from the first case; the latter two are asymptotically negligible.
\paragraph{Case 1: $\inf_{x \in Q} f(x) < 2^{-js}$}
Under the assumption that $f \in \cC^s$, we have that $\sup_{x \in Q} \leq \inf_{x \in Q} + \diam(Q)^s \lesssim 2^{-js}$.
Therefore $\mu(Q) = \int_Q f(x) \dd x \lesssim 2^{-j(d+s)}$.

To evaluate $K_{j+1}(f - g)$, we note that for each $R \subseteq Q, R \in \cQ_{j+1}$, the densities $K_{j+1} f$ and $K_{j+1} g $ are constant on R, with value $2^{d(j+1)} \mu(R)$ and $2^{d(j+1)} \nu(R)$.
This implies
\begin{align*}
\|K_{j+1}(f - g)\|_{L_1(Q)} & = \sum_{R \subseteq Q, R \in \cQ_{j+1}} 2^{d(j+1)} \int_R |\mu(R) - \nu(R)| \\
& = \sum_{R \subseteq Q, R \in \cQ_{j+1}} |\mu(R) - \nu(R)|\,,
\end{align*}
In particular, recalling that $n \cdot \nu(R) \sim \mathrm{Bin}(n, \mu(R))$, we have that that
\begin{align}\label{l1_variance_bound}
\E \|K_{j+1}(f - g)\|_{L_1(Q)}^2 \lesssim \sum_{\substack{R\subseteq Q \\ R \in \cQ_{j+1}}} \E (\mu(R) - \nu(R))^2 \leq \mu(Q)/n\,.
\end{align}
Jensen's inequality along with the fact that $\sup_{x \in Q} f(x) \leq 2^{-js}$ therefore yields
\begin{equation*}
\E \|K_{j+1}(f - g)\|_{L_1(Q)} \leq (\mu(Q)/n)^{1/2} \lesssim 2^{-j(d+s)/2} n^{-1/2}\,.
\end{equation*}

\paragraph{Case 2: $\inf_{x \in Q} f(x) \geq 2^{-js}$ and $\cE_Q$ holds}
Under the assumption that $f \in \cC^s$, we have that $\sup_{x \in Q} f(x) - \inf_{x \in Q} f(x) \leq \diam(Q)^s \lesssim 2^{-js}$.
We therefore have $\Delta_Q \lesssim 1$.
Moreover, on $\cE_Q$, we have
\begin{equation*}
m_Q \gtrsim \inf_{x \in Q} f(x)\,.
\end{equation*}

Arguing as above, we have
\begin{align*}
\|K_{j+1}(f - g)\|_{L_p(Q)}^p & = \sum_{R \subseteq Q, R \in \cQ_{j+1}} 2^{dp(j+1)} \int_R |\mu(R) - \nu(R)|^p \dd x \\
& \lesssim 2^{dj(p-1)} \sum_{R \subseteq Q, R \in \cQ_{j+1}} |\mu(R) - \nu(R)|^p\,.
\end{align*}

Since $\sup_{x \in Q} f(x) \asymp \inf_{x \in Q} f(x)$, we have $\mu(R) \asymp 2^{-jd} \inf_{x \in Q} f(x)$.
Since $n \gtrsim 2^{J(d+s)} \geq 2^{j(d+s)} \geq 1/\mu(R)$, we can apply \cref{binomial_moments} to obtain
\begin{equation*}
\E |\mu(R) - \nu(R)|^p \lesssim (\mu(R)/n)^{p/2} \lesssim (2^{-jd}\inf_{x \in Q} f(x)/n)^{p/2}\,.
\end{equation*}

Combining these bounds, we obtain
\begin{align*}
\E[m_Q^{1-p} \Delta_Q^p \|K_{j+1}(f - g)\|^p_{L_p(Q)} \indic{\cE_Q}] & \lesssim 2^{jd(\frac p2 - 1)} (\inf_{x \in Q} f(x))^{1-\frac p2} n^{-p/2} \\
& \leq 2^{j(d+s)(\frac p2 - 1)} n^{-p/2}\,.
\end{align*}

\paragraph{Case 3: $\cE_Q$ does not hold}
The Cauchy-Schwarz inequality and the bound~\eqref{l1_variance_bound} imply
\begin{align*}
\E[\|K_{j+1}(f - g)\|_{L_1(Q)} \indic{\cE_Q^C}] & \leq (\E \|K_{j+1}(f - g)\|^{2}_{L_1(Q)})^{1/2} \p[\cE_Q^C]^{1/2} \\
& \lesssim (\mu(Q)/n)^{1/2} \p[\cE_Q^C]^{1/2}\,.
\end{align*}
By \cref{binomial_deviations}, 
\begin{align*}
\p[\cE_Q^C] & \leq e^{- n\mu(Q)/8} \lesssim \frac{1}{n \mu(Q)}
\end{align*}
and combining these bounds yields
\begin{equation*}
\E[\|K_{j+1}(f - g)\|_{L_1(Q)} \indic{\cE_Q^C}] \lesssim n^{-1}\,.
\end{equation*}

\medskip
We are now prepared to bound $W_p(f, g)$.
Combining the three cases above, we observe
\begin{align*}
\E \left\{\|K_{j+1}(f - g)\|_{L_1(Q)} \wedge m_{Q}^{1-p} \Delta_Q^p \|K_{j+1}(f - g)\|_{L_p(Q)}^p\right\} & \lesssim 2^{-j(d+s)/2} n^{-1/2} \\
& + 2^{j(d+s)(\frac p2 - 1)} n^{-p/2} \\
&+ n^{-1}\,.
\end{align*}

Under the assumption that $n \geq 2^{j(d+s)}$, the first term dominates.
We therefore have
\begin{equation*}
\E \sum_{Q \in \cQ^j} \|K_{j+1}(f - g)\|_{L_1(Q)} \wedge m_{Q}^{1-p} \Delta_Q^p \|K_{j+1}(f - g)\|_{L_p(Q)}^p \lesssim 2^{j(d-s)/2} n^{-1/2}\,,
\end{equation*}
and \cref{l1_lp_bound} combined with Jensen's inequality implies
\begin{equation*}
\E W_p(f, g) \lesssim \sum_{0 \leq j < J} 2^{-j} 2^{j(d-s)/2p} n^{-1/2p} \lesssim  \left\{\begin{array}{ll}
n^{- \frac{1 + s/p}{d+s}} & \text{if $d - s > 2p$} \\
n^{-1/2p} \log n & \text{if $d - s = 2p$} \\
n^{-1/2p} & \text{if $d - s < 2p$.}
\end{array}\right.
\end{equation*}

On the other hand, a simple truncation argument (\cref{unbounded_truncation}, below) shows that $\E W_p(g, \hat f) \lesssim 2^{-J(1 + s/p)} \asymp n^{-\frac{1+s/p}{d + s}}$.

Since $W_p(f, \hat f) \leq W_p(f, g) + W_p(g, {\hat f})$, combining these two bounds yields the claim.
\qed

All that remains is to give the proof of \cref{unbounded_truncation}.
\begin{lemma}\label{unbounded_truncation}
For $g$ and $\hat f$ defined as in the proof of \cref{thm:unbounded_ub}, we have
\begin{equation*}
\E W_p(g, {\hat f}) \lesssim 2^{-J(1 + s/p)}\,.
\end{equation*}
\end{lemma}
\begin{proof}
By construction, $\hat f$ is an element of $V_J$.
Moreover, $g$ is defined so that $K_J g = K_J \hat f$.
Combining these two facts yields that $\hat f = K_J g$.
For $j \leq J$, this implies that $K_j \hat f = K_j g$, and for $j > J$, it yields
\begin{equation*}
\|K_{j}(\hat f - g)\|_{L_1(Q)} = \|K_J g - K_j g\|_{L_1(Q)} = \frac{\nu(Q)}{\mu(Q)} \|K_J f - K_j f\|_{L_1(Q)}\,,
\end{equation*}
where in the last step we use the definition of $g$.
Since $\E \nu(Q) = \mu(Q)$, we obtain
\begin{equation*}
\E \|K_{j}(\hat f - g)\|_{L_1(\Omega)} = \left\{\begin{array}{ll}
0 & \text{if $j \leq J$} \\
\|(K_j - K_J) f\|_{L_1(\Omega)} & \text{ if $j > J$.}
\end{array}\right.
\end{equation*}

For $j > J$, the operator $K_j - K_J$ is a projection onto the span of $\{\Psi_k\}_{J \leq k < j}$.
Under the assumption that $f \in \cC^s = \cB_{\infty, \infty}^s$, we obtain
\begin{align*}
\|(K_j - K_J) f\|_{L_1(\Omega)} & = \left\|\sum_{J \leq k < j} \sum_{\psi \in \Psi_k} \beta_\psi \psi \right\|_{L_1(\Omega)} \\
& \leq \sum_{J \leq k < j} \left\|\sum_{\psi \in \Psi_k} \beta_\psi \psi \right\|_{L_1(\Omega)} \\
& \leq \sum_{J \leq k < j} 2^{-dk/2} \|\beta\|_{\ell_1} \lesssim 2^{-Js}\,,
\end{align*}
where the final inequality uses \cref{assume:stability}.

Therefore, applying the bound of~\citet{WeeBac18} in the equivalent form~\eqref{eq:projection_expansion}, we obtain
\begin{align*}
\E W_p(\hat f, g) &\lesssim \sum_{j \geq 0} 2^{-j} \E \|K_j(\hat f - g)\|_{L_1(\Omega)}^{1/p} \\
& \leq \sum_{j \geq 0} 2^{-j} (\E \|K_j(\hat f - g)\|_{L_1(\Omega)})^{1/p} \\
& = \sum_{j > J} 2^{-j} \|(K_j - K_J) f\|_{L_1(\Omega)}^{1/p} \lesssim 2^{-J(1 + s/p)}\,,
\end{align*}
as desired.
\end{proof}

\section{Proof of Theorem 10}\label{app:refined} 
The assumption that $g/f \in V_J$ implies that, for any $j \leq J$, if $\int_Q f = 0$, then $\int_Q g = 0$.
For $0 \leq j \leq J$, we therefore define $f_j$ by setting
\begin{equation*}
f_j(x) = \frac{\int_Q g}{\int_Q f} f(x) \quad \forall x \in Q, Q \in \cQ_j\,,
\end{equation*}
where we interpret $0/0 = 0$.
With this definition, $f_0 = f$ and $f_J = g$.
Moreover, for $0 \leq j \leq J$, we have $K_j f_j = K_j f_{j+1}$.

The triangle inequality therefore implies
\begin{equation}\label{eq:triangle_wp}
W_p(f, g) \leq \sum_{0 \leq j < J-1} W_p({f_j}, {f_{j+1}})\,.
\end{equation}
We focus on bounding the quantity $W_p({f_j}, {f_{j+1}})$, from which the claimed bound will follow.

We require the following lemma, whose proof is deferred.
\begin{lemma}\label{prop:best_of_both}
Fix $j \geq 0$.
Let $p \geq 1$, and let $f_j$ and $f_{j+1}$ be any two densities such that $K_j f_j = K_j f_{j+1}$.
Then
\begin{equation*}
W_p({f_j}, {f_{j+1}}) \lesssim 2^{-j} \left(\sum_{Q \in \cQ_j} \|f_j - f_{j+1}\|_{L_1(Q)} \wedge (\inf_{x \in Q} f_j(x))^{1-p} \|f_j-f_{j+1}\|^p_{L_p(Q)} \right)^{1/p}\,.
\end{equation*}
\end{lemma}

By the definition of $f_j$, we have $\inf_{x \in Q} f_j(x) = \frac{\int_Q g}{\int_Q f} \inf_{x \in Q} f(x) = m_Q$.
The claim then follows from the following lemma.
\begin{lemma}\label{l1_lp_projection}
For any $Q \in \cQ_j$, the densities $f_j$ and $f_{j+1}$ satisfy
\begin{align*}
\|f_{j} - f_{j+1}\|_{L_1(Q)} & \lesssim \|K_{j+1}(f - g)\|_{L_1(Q)} \\
\|f_{j} - f_{j+1}\|_{L_p(Q)} & \lesssim \Delta_Q \|K_{j+1}(f - g)\|_{L_p(Q)} \quad \forall p \geq 1
\end{align*}
\end{lemma}
Combining~\eqref{eq:triangle_wp}, \cref{prop:best_of_both}, and~\cref{l1_lp_projection} yields \cref{l1_lp_bound}. \qed

It remains to give proofs of the two required lemmas.
\subsection{Proof of \cref{prop:best_of_both}}
If $K_j f_j = K_j f_{j+1}$, then the measures ${f_j}$ and ${f_{j+1}}$ assign the same mass to each cube in $\cQ_j$.
We can therefore obtain a coupling between ${f_j}$ and ${f_{j+1}}$ by optimally coupling ${f_j}|_Q$ and ${f_{j+1}}|_Q$ on each $Q \in \cQ_j$ and combining the resulting couplings~\citep[see][Lemma A.2]{WeeBac18}.
Doing so yields the bound
\begin{equation*}
W_p^p({f_j}, {f_{j+1}}) \leq \sum_{Q \in \cQ_j} W_p^p({f_j}|_Q, {f_{j+1}}|_Q)\,.
\end{equation*}

To show that $W_p^p({f_j}|_Q, {f_{j+1}}|_Q) \leq 2^{-kp} \|f_j - f_{j+1}\|_{L_1(Q)}$, we use the fact that for two measures with support in a set $K$, it holds $W_p^p(\mu, \nu) \leq \diam(K)^p d_{\mathrm{TV}}(\mu, \nu)$, and when $\mu$ and $\nu$ have densities, the total variation $d_{\mathrm{TV}}$ reduces to the $L_1$ norm.

To obtain the second bound, we apply \cref{prop:wavelet_ub}.
By \cref{assume:stability}, if we denote by $P_j$ the $L_2$ projection onto the span of $\Psi_j$, then the bound of \cref{prop:wavelet_ub} reads
\begin{align*}
W_p({f_j}|_Q, {f_{j+1}}|_Q) & \lesssim (\inf_{x \in Q} f_j(x))^{-1/p'} \sum_{k \geq j} 2^{-k} \|P_j(f_j - f_{j+1})\|_{L_p(Q)} \\
& \lesssim 2^{-j} (\inf_{x \in Q} f_j(x))^{-1/p'} \|f_j - f_{j+1}\|_{L_p(Q)}\,,
\end{align*}
where the second inequality follows from \cref{assume:projection_stability}.
\qed

\subsection{Proof of \cref{l1_lp_projection}}
For convenience, let us write $\mu$ for the measure with density $f$ and $\nu$ for the measure with density $g$.
We recall that on $R \in \cQ_{j+1}$, the densities $K_{j+1} f$ and $K_{j+1} g$ are constant with values $2^{d(j+1)} \mu(R)$ and $2^{d(j+1)} \nu(R)$.
For any cube $R \in \cQ_{j+1}$ such that $R \subseteq Q$, we have that $(f_j - f_{j+1}) = \left(\frac{\nu(Q)}{\mu(Q)} - \frac{\nu(R)}{\mu(R)}\right) f$ on $R$, which implies
\begin{align*}
\|f_j - f_{j+1}\|_{L_1(Q)} & = \sum_{\substack{R \in \cQ_{j+1} \\ R \subseteq Q}} \|f\|_{L_1(R)} \left| \frac{\nu(Q)}{\mu(Q)} - \frac{\nu(R)}{\mu(R)}\right|\\
\|f_j - f_{j+1}\|_{L_p(Q)}^p & = \sum_{\substack{R \in \cQ_{j+1} \\ R \subseteq Q}} \|f\|_{L_p(R)}^p \left| \frac{\nu(Q)}{\mu(Q)} - \frac{\nu(R)}{\mu(R)}\right|^p\,.
\end{align*}
The first sum is bounded by $|\nu(Q) - \mu(Q)| + \sum_R |\nu(R) - \mu(R)| \leq 2 \sum_{R} |\nu(R) - \mu(R)| = 2 \|K_{j+1} (f - g)\|_{L_1(Q)}$.

For the second, we write
\begin{equation*}
\left| \frac{\nu(Q)}{\mu(Q)} - \frac{\nu(R)}{\mu(R)}\right|^p \lesssim \frac{1}{\mu(R)^p} |\nu(R) - \mu(R)|^p + \frac{1}{\mu(Q)^p} |\nu(Q) - \mu(Q)|^p\,.
\end{equation*}
Since $R$ is a cube of side length $2^{-(j+1)}$, we have the simple bounds
\begin{align*}
\mu(Q)^p \geq \mu(R)^p & \geq 2^{-(j+1)pd} (\inf_{x \in Q} f(x))^p \\
\|f\|_{L_p(R)}^p & \leq 2^{-(j+1)d} (\sup_{x \in Q} f(x))^p\,.
\end{align*}
We obtain
\begin{equation*}
\|f\|_{L_p(R)}^p \left| \frac{\nu(Q)}{\mu(Q)} - \frac{\nu(R)}{\mu(R)}\right|^p \lesssim \Delta_Q^p 2^{(j+1)(p-1)d}(|\nu(R) - \mu(R)|^p + |\nu(Q) - \mu(Q)|^p)\,,
\end{equation*}
and applying the definition of $K_{j+1} f$ and $K_{j+1} g$ yields
\begin{equation*}
\|f_j - f_{j+1}\|_{L_p(Q)}^p = \sum_{\substack{R \in \cQ_{j+1} \\ R \subseteq Q}} \|f\|_{L_p(R)}^p \left| \frac{\nu(Q)}{\mu(Q)} - \frac{\nu(R)}{\mu(R)}\right|^p \lesssim \Delta_Q^p \|K_{j+1}(f - g)\|_{L_p(Q)}^p\,,
\end{equation*}
as desired.
\qed

\section{Assumptions on wavelets}\label{app:wave}
For completeness, we have extracted the properties we require of the sets $\Phi$ and $\Psi_j$.
\begin{assume}[Basis]\label{assume:basis}
$\Phi \cup \left\{\bigcup_{j \geq 0} \Psi_{j}\right\}$ forms an orthonormal basis for $L_2([0, 1]^d)$. 
\end{assume}
\begin{assume}[Regularity]\label{assume:regularity}
The functions in $\Phi$ and $\Psi_j$ for $j \geq 0$ all lie in $\cC^r(\Omega)$, and polynomials of degree at most $r$ on $\Omega$ lie in $\mathrm{span}(\Phi)$.
\end{assume}
\begin{assume}[Tensor construction]\label{assume:tensor}
Each $\psi \in \Psi_j$ can be expressed as $\psi(\mathbf{x}) = \prod_{i=1}^d \psi_i(x_i)$ for some univariate functions $\psi_i$.
\end{assume}
\begin{assume}[Locality]\label{assume:locality}
For each $\psi \in \Psi_j$ there exists a rectangle $I_\psi \subseteq [0, 1]^d$ such that $\supp(\psi) \subseteq I_\psi$, $\diam(I_\psi) \lesssim 2^{-j}$, and $\left\|\sum_{\psi \in \Psi_j} \1\{x \in I_\psi\}\right\|_{L_\infty} \lesssim 1$.
\end{assume}
\begin{assume}[Norm]\label{assume:norm}
$\|\psi\|_{L_p(\Omega)} \asymp 2^{dj(\frac 12 - \frac 1p )}$ for all $\psi \in \Psi_j$.
\end{assume}
\begin{assume}[Bernstein estimate]\label{assume:bernstein}
$\|\nabla f\|_{L_p(\Omega)} \lesssim 2^{j} \|f\|_{L_p(\Omega)}$ for any $f \in \mathrm{span}\left(\Phi \cup \left\{\bigcup_{0 \leq k < j} \Psi_k\right\}\right)$.
\end{assume}

These requirements are achievable for $\Omega = [0, 1]^d$ due to a classic construction due to~\citet{CohDauVia93}. \citep[See also][Chapter~2.12, for further details.]{Coh03}

We state without proof some straightforward consequences of our assumptions.
\begin{lemma}\label{assume:size}
$|\Phi| \lesssim 1$ and $|\Psi_j| \lesssim 2^{dj}$ for $j \geq 0$.
\end{lemma}

\begin{lemma}\label{assume:stability}
For any vector $\{\alpha_\phi\}_{\phi \in \Phi}$, $\|\sum_{\phi \in \Phi} \alpha_\phi \phi\|_{L_p(\Omega)} \asymp \|\alpha\|_{\ell_p}$.
Likewise, for any $\{\beta_\psi\}_{\psi \in \Psi_j}$, $\|\sum_{\psi \in \Psi_j} \beta_\psi \psi\|_{L_p(\Omega)} \asymp 2^{dj(\frac 12 - \frac 1p )} \|\beta\|_{\ell_p}$.
\end{lemma}

\begin{lemma}\label{assume:projection_stability}
If $P_j$ denotes the orthogonal projection onto the span of $\Psi_j$, then $\|P_j f\|_{L_p(\Omega)} \lesssim \|f\|_{L_p(\Omega)}$.
\end{lemma} 

\bibliographystyle{abbrvnat_weed}

\end{document}